\documentclass{amsart}
\usepackage{enumitem,kantlipsum}

\usepackage[utf8]{inputenc}
\usepackage{physics}
\usepackage{amssymb}
\usepackage{amsmath}
\usepackage{hyperref}
\usepackage{amsfonts}
\usepackage{amsthm}
\usepackage{mathtools}
\usepackage{tikz}
\usepackage{tikz-cd}
\usepackage{algorithmic}
\usepackage{asymptote}
\usepackage{url}
\usepackage{colonequals} 
\usepackage{amsmath,amssymb}
\makeatletter
\newsavebox\myboxA
\newsavebox\myboxB
\newlength\mylenA

\hypersetup{colorlinks=true}

\def\equationautorefname~#1\null{%
  Equation~(#1)\null
}

\newcounter{itemadded}
\newtheorem{theorem}{Theorem}[section]
\newtheorem{corollary}{Corollary}[theorem]
\newtheorem{lemma}[theorem]{Lemma}
\newtheorem{proposition}[theorem]{Proposition}

\newtheorem{definition}[theorem]{Definition}

\usepackage[OT2,T1]{fontenc}
\DeclareSymbolFont{cyrletters}{OT2}{wncyr}{m}{n}
\DeclareMathSymbol{\Sha}{\mathalpha}{cyrletters}{"58}

\DeclarePairedDelimiter\floor{\lfloor}{\rfloor}

\theoremstyle{definition}

\newtheorem{example}[theorem]{Example}
\newtheorem{remark}[theorem]{Remark}

\newtheorem{notation}[theorem]{Notation}

\newcommand*\xoverline[2][0.75]{%
    \sbox{\myboxA}{$\m@th#2$}%
    \setbox\myboxB\null
    \ht\myboxB=\ht\myboxA%
    \dp\myboxB=\dp\myboxA%
    \wd\myboxB=#1\wd\myboxA
    \sbox\myboxB{$\m@th\overline{\copy\myboxB}$}
    \setlength\mylenA{\the\wd\myboxA}
    \addtolength\mylenA{-\the\wd\myboxB}%
    \ifdim\wd\myboxB<\wd\myboxA%
       \rlap{\hskip 0.5\mylenA\usebox\myboxB}{\usebox\myboxA}%
    \else
        \hskip -0.5\mylenA\rlap{\usebox\myboxA}{\hskip 0.5\mylenA\usebox\myboxB}%
    \fi}
\makeatother
\newcommand{\Z}{\mathbb{Z}}
\newcommand{\ab}{\underline{a}}
\newcommand{\ml}{\mu_m}
\newcommand{\smvee}{\raise0.4ex\hbox{$\scriptscriptstyle\vee$}}

\newcommand{\ag}{\mathcal{A}_g}
\newcommand{\mg}{\mathcal{M}_g}
\newcommand{\OO}{\mathcal{O}}
\newcommand{\p}{\mathfrak{p}}

\newcommand{\Q}{\mathbb{Q}}

\newcommand{\Spec}{\text{Spec}}

\newcommand{\im}{\text{Im}}

\newcommand{\C}{\mathbb{C}}
\newcommand{\fpb}{\overline{\mathbb{F}}_p}
\newcommand{\sigp}{\sigma_p}

\newcommand{\Po}{\mathbb{P}^1}

\newcommand{\mmg}{\mathcal{M}_{G}}

\newcommand{\lra}{(m,r,\underline{a})}

\newcommand{\zgr}{\mathcal{M}(G,r,\underline{a})}
\newcommand{\gra}{(G,r,\underline{a})}

\newcommand{\shmf}{\text{Sh}(\ml,\underline{f})}
\newcommand{\shgf}{\text{Sh}(G,\underline{f})}

\newcommand{\go}{\gamma_1}
\newcommand{\gt}{\gamma_2}
\newcommand{\gth}{\gamma_3}

\newcommand{\ts}{\tau^*}
\newcommand{\wt}{w_{\tau}}
\newcommand{\tp}{\tau'}
\newcommand{\wts}{w_{\tau^*}}

\newcommand{\ti}{\tau_i}

\newcommand{\eis}{\eta_{i,s}}

\newcommand{\tis}{\ti^*}

\newcommand{\psid}{\check{\psi}}
\newcommand{\phid}{\check{\phi}}

\newcommand{\hy}{H^1(C,\OO_C)}
\newcommand{\hyd}{H^0(C,\Omega_C)}

\newcommand{\pby}{\floor*{p\langle\frac{ia_1}{m}\rangle}}
\newcommand{\pbr}{\floor*{p\langle\frac{ia_r}{m}\rangle}}
\newcommand{\pbj}{\floor*{p\langle\frac{ia_j}{m}\rangle}}

\newcommand{\pbis}{\floor*{p\langle\frac{b_i a_s}{m}\rangle}}
\newcommand{\pbc}{\floor*{p\langle\frac{ia_c}{m}\rangle}}
\newcommand{\pbcm}{\floor*{p\langle\frac{ia_{c-1}}{m}\rangle}}

\newcommand{\pbk}{\floor*{p\langle\frac{ia_k}{m}\rangle}}

\newcommand{\pfy}{\floor*{p\frac{ia_1}{m}}}
\newcommand{\pfk}{\floor*{p\frac{ia_k}{m}}}
\newcommand{\fy}{\floor*{\frac{ia_1}{m}}}
\newcommand{\fr}{\floor*{\frac{ia_r}{m}}}
\newcommand{\fk}{\floor*{\frac{ia_k}{m}}}
\newcommand{\fky}{\floor*{\frac{ka_1}{m}}}
\newcommand{\fkr}{\floor*{\frac{ka_r}{m}}}
\newcommand{\pfr}{\floor*{p\frac{ia_r}{m}}}

\newcommand{\ci}{c_i(s_i-pj)}
\newcommand{\Ci}{C_i(s_i-pj)}
\newcommand{\XI}{X_i(s_i-pj)}

\newcommand{\qi}{Q_i}

\newcommand{\qt}{Q_\tau}

\newcommand{\qtsd}{Q_{\ts}^{\smvee}}

\newcommand{\qpt}{Q_{p\tau}}

\newcommand{\qptsd}{Q_{p\ts}^{\smvee}}
\newcommand{\ocu}{\OO_C(U)}
\newcommand{\ocv}{\OO_C(V)}
\newcommand{\ocuv}{\OO_C(U \cap V)}
\newcommand{\zij}{\zeta_{i,j}}
\newcommand{\zpij}{\zeta_{pi,j'}}
\newcommand{\zpisk}{\zeta_{pi^*,k}}

\newcommand{\opisjp}{\omega_{pi^*,j'}}
\newcommand{\opisk}{\omega_{pi^*,k}}

\newcommand{\phii}{\phi_{\tau_i}}
\newcommand{\psii}{\psi_{\tau_i}}
\newcommand{\psiip}{\psi_{\tau_i}'}

\newcommand{\vip}{v_{pi}}

\newcommand{\vips}{v_{pi^*}}
\newcommand{\pbf}{(x-x_1)^{\pby} \dots (x-x_r)^{\pbr}}

\newcommand{\fis}{f(\tau_i^*)}

\newcommand{\ft}{f(\tau)}
\newcommand{\fts}{f(\ts)}
\newcommand{\pit}{p^i\tau}
\newcommand{\pits}{p^i\ts}

\newcommand{\aip}{A_{l-1}\circ \dots \circ A_{0}}

\newcommand{\Span}{\text{Span}}

\newcommand{\au}{\underline{a}}
\newcommand{\n}{\text{new}}
\newcommand{\ord}{\text{ord}}

\newcommand{\gal}{\text{Gal}}
\newcommand{\galq}{\gal(\overline{\Q}/\Q)}
\newcommand{\fb}{\underline{f}}

\newcommand{\og}{\OO_G}

\newcommand{\Hom}{\text{Hom}}

\newcommand{\orho}{\OO_{\tau}}
\newcommand{\tg}{\mathcal{T}_G}

\newcommand{\kij}{\chi_{i,j}}
\newcommand{\ct}{C_{\underline{x}}}
\newcommand{\mlra}{\mathcal{M}(m,r,\au)}
\newcommand{\mgra}{\mathcal{M}(G,r,\au)}
\newcommand{\mgb}{\overline{\mathcal{M}_G}}
\newcommand{\mG}{\mathcal{M}_G}

\newcommand{\tb}{\underline{x}}
\newcommand{\ub}{\underline{b}}

\newcommand{\xb}{\underline{x}}

\newcommand{\tguh}{\tg^H}
\newcommand{\tghn}{\tg^{H,\text{new}}}
\newcommand{\tgun}{\tg^{\text{new}}}
\newcommand{\MM}{\mathcal{M}}

\newcommand{\gb}{\gamma_b}
\newcommand{\mts}{\mathcal{M}(\gamma_{t,s})}
\newcommand{\gts}{\gamma_{t,s}}

\usepackage{enumitem,kantlipsum}

\newcommand{\FF}{\mathcal{F}}
\usepackage{amsmath}
\usepackage{kbordermatrix}
\renewcommand{\kbldelim}{(}
\renewcommand{\kbrdelim}{)}

\begin{document}

\title{Abelian covers of $\mathbb{P}^1$ of $p$-ordinary Ekedahl-Oort type}

\author{Yuxin Lin}
\address{Department of Mathematics,
California Institute of Technology
Pasadena, CA 91125, USA}
\email{yuxinlin@caltech.edu}

\author{Elena Mantovan}
\address{Department of Mathematics,
California Institute of Technology
Pasadena, CA 91125, USA}
\email{mantovan@caltech.edu}

\author{Deepesh Singhal}
\address{Department of Mathematics,
University of California, Irvine
Irvine, CA 92697, USA}
\email{singhald@uci.edu}

\date{}
\maketitle 
\begin{abstract}
Given a family of abelian covers of $\Po$ and a prime $p$ of good reduction, by considering the associated Deligne--Mostow Shimura variety, we obtain lower bounds for the Ekedahl-Oort types, and the Newton polygons, at $p$ of the curves in the family. In this paper, we investigate whether such lower bounds are sharp. In particular, we prove sharpeness when the number of branching points is at most five and $p$ sufficiently large. Our result is a generalization under stricter assumptions of \cite[Theorem 6.1]{Irene} by Bouw, which proves the analogous statement for the $p$-rank, and it relies on the notion of Hasse-Witt triple introduced by Moonen in \cite{Moonen algorithm}.  
\end{abstract}

\section{Introduction}
This paper is motivated by the arithmetic Schottky problem in positive characteristics, which investigates which mod-$p$ invariants of abelian varieties occur for Jacobians  of smooth curves. We restrict our attention to the case of Jacobians of abelian covers of $\Po$.

Let $G$ be a finite abelian group and $p$ be a prime not dividing $ |G|$. We consider $\mmg$, the moduli space of $G$-covers of $\Po$. On each irreducible component of $\mmg$, the monodromy datum $(G,r,\ab)$ of the covers is constant, and we denote such component by $\mathcal{M}(G,r,\ab)$ (in the notation $(G,r, \ab)$, $r$ denotes the number of branched points and $\ab$ the inertia type). 
By construction, the image of $\mgra$ under the Torelli map $T$ is contained in a special subvariety of the Siegel variety. We denote by $\shgf$  the smallest PEL-type Shimura variety containing $T(\mgra)$; its Shimura datum depends on the monodromy datum $\gra$.

The inclusion of $T(\zgr)$ inside $\shgf$ gives rise to natural lower bounds for the Ekedahl-Oort types and the Newton polygons occurring for the Jacobians of curves parametrized by $\mathcal{M}(G,r,\ab)$.
More precisely, let $p$ be a prime not dividing $|G|$. Then, both $\mgra$ and $\shgf$ have good reduction at $p$, and the $p$-rank, Ekedahl–Oort and Newton stratifications of $\mgra_{\fpb}$ are induced from those on $\shgf_{\fpb}$. In particular, the maximal $p$-rank and the lowest Ekedahl–Oort type and Newton polygon occurring on $\shgf_{\fpb}$  are respectively upper and lower bounds for those occurring on $\mgra_{\fpb}$. For example, by \cite[Theorem 1.6.3]{Wedhorn}, when $p$ is not totally split in the reflex field of $\shgf_{\fpb}$, the ordinary stratum of $\shgf_{\fpb}$ is empty, and hence so is that of $\mgra_{\fpb}$.  It is natural to ask whether these bounds are sharp. More precisely, given a monodromy datum $\gra$ and a prime $p\nmid |G|$, one may ask when the intersection of $T(\zgr)$ with each of the unique open strata of $\shgf_{\fpb}$ is non-empty. When $r\leq 3$, $\dim\zgr=\dim\shgf=0$, and the statement is trivial. On the other hand, by the Coleman--Oort Conjecture, if $r\geq 4$, $\dim\zgr < \dim\shgf$ except in finitely many instances (see \cite{Moonen 10} and \cite{Moonen Oort}). 

For $p$ large, the sharpness of the $p$-rank bound follows as a special case of \cite[Theorem 6.1]{Irene}. 
In this paper, we investigate the sharpness of the lower bounds for Ekedahl-Oort types and Newton polygons, and positively answer our question when $p$ is large and the number of branched points is at most $5$.
By \cite[Theorem 1.3.7]{Moonen EO type formula},  the unique open Ekedahl–Oort and Newton strata of $\shgf_{\fpb}$ agree, hence the two problems are equivalent. 
  On the other hand, the open Newton/Ekedahl-Oort stratum is in general, properly contained in the maximal $p$-rank stratum. Thus our result is a refinement of the special case of \cite[Theorem 6.1]{Irene} for covers of $\Po$ branched at at most 5 points.

More precisely, we prove the following result. For $p\nmid |G|$, we refer to the lowest Ekedahl--Oort type and Newton polygon occurring on $\shgf_{\fpb}$ as $(G,\underline{f})$-ordinary.
\begin{theorem}\label{abelian cover}
Let $(G,r,\underline{a})$ be a monodromy datum for abelian $G$-covers of $\Po$, branched at $r$ points, and $p$ a rational prime.  Assume $r\leq 5 $ and $p>|G|(r-2)$.
Then the Ekedahl–Oort type and Newton polygon of the generic $G$-cover of $\Po$ over $\fpb$, with monodromy datum $(G,r,\underline{a})$, are $(G,\underline{f})$-ordinary.
\end{theorem}

The condition  $p>|G|(r-2)$ is the same as in \cite[Theorem 6.1]{Irene}. 
The restriction $r\leq 5$ is due to the complexity of the computations. By combining Theorem \ref{abelian cover}  with the results in \cite[Section 4]{clutching argument}, for any $r\geq 6$ we construct infinitely many examples of monodromy data with $r$ branched points that satisfy the statement of Theorem \ref{abelian cover} (see Remark \ref{anyrl}). 

As an application, Theorem \ref{abelian cover} yields new examples
of Newton polygons that occur for Jacobians of smooth curves, and
of unlikely intersections of the Torelli locus with Newton strata in Siegel varieties (see Remarks \ref{newNP} and \ref{unlikely}).

We describe our strategy. 
In \cite{Irene}, Bouw considers ramified abelian prime-to-$p$ covers of curves over $\fpb$, and the natural upper bound for their $p$-rank coming from the Galois action. By expressing the $p$-rank of a cover in terms of its Hasse--Witt invariants (the ranks of certain blocks in the Hasse--Witt matrix of the curve), in \cite[Theorem 6.1]{Irene}, Bouw proves that the generic cover has $p$-rank equal to the upper bound. When specialized to covers of $\Po$, the upper bound for the $p$-rank in \cite{Irene} agrees with the $(G,\underline{f})$-ordinary $p$-rank. Recently, in \cite{Moonen algorithm}, Moonen introduces the notion of Hasse--Witt triple for a smooth curve over $\fpb$, as a generalization of the Hasse--Witt invariant which arises from a suitable extension of the Hasse--Witt matrix. In \cite[Theorem 2.8]{Moonen algorithm} Moonen proves that the new notion is equivalent to the Ekedahl-Oort type of the Jacobian of the curve.   
In this paper, we prove that for the generic curve in the family, the Hasse--Witt triple is $(G,\underline{f})$-ordinary by expressing the condition of $(G,\underline{f})$-ordinariness of a cover as the non-vanishing of certain minors of (iterations of) its extended Hasse--Witt matrix.
 Our approach and computations are modeled on those in \cite{Irene}.

The paper is organized as follows. 
In Section \ref{pre}, we recall the notions of $(G,\underline{f})$-ordinary Ekedahl--Oort type and Newton polygon,  and Moonen's notion of Hasse-Witt triple. In Section \ref{reduction argument}, we reduce the proof of Theorem \ref{abelian cover} to the case when $G$ is a cyclic group (Lemma \ref{equivalence two}). In Section \ref{combi}, we give a criterion for $(G,\underline{f})$-ordinaryness in terms of Hasse--Witt triples, given as (finitely many) rank conditions on iterations of the extended Hasse-Witt matrix (Theorem \ref{numerical criteria}). In Section \ref{duality} and \ref{max monomial}, we determine the Hasse-Witt triple of a cyclic cover of the projective line\footnote{
In \cite{Moonen algorithm}, Moonen gives an algorithm for computing the Hasse--Witt triple of a smooth curve given as a complete intersection; this algorithm does not apply in our context (see Remark \ref{strategy}).}, and show the non vanishing of certain entries of the associated extended Hasse--Witt matrix.
In Sections \ref{Section 0,1 in signature} and \ref{new part mu ordinary}, we prove Theorem \ref{abelian cover}, first under some additional conditions, and then in general.

\section{Notations and Preliminaries}\label{pre}

\subsection{The Hurwitz space $\mgra$} 
We briefly recall the definition of the Hurwitz space of abelian covers. We refer to \cite{moduli functor G cover} for a more complete description of the construction of the moduli functor.

\begin{definition}\label{monodromy datum} 
Let $G$ be a finite group.
A monodromy datum of a $G$-cover of $\Po$ is given as $(G,r,\au)$, where $r\geq 3$ is the number of branched points, and $\au\in G^r$ is the inertia type of the cover.
That is, $\au=(\au_1,\au_2,\dots, \au_r)\in G^r$ satisfies 
\begin{enumerate}
    \item $\au_i \neq 0$ in $G$,
    \item $\au_1, \dots, \au_r$ generate $G$,
    \item $\sum_{i=1}^r \au_i = 0$ in $G$.
\end{enumerate}
\end{definition}

Let $\mg$ be the moduli space of smooth projective genus $g$ curves and $\overline{\mg}$ be its Deligne-Mumford compactification, so it is the moduli space of stable curves of genus $g$. While both are algebraic stacks defined over $\Q$, they have a model defined over some open subset of $\Spec(\Z)$.

Let $p$ be a rational prime, $p \nmid |G|$. We use $e$ to denote the exponent of $G$ and consider schemes over $\Z[\frac{1}{e}, \zeta_e]$, where $\zeta_e$ is a primitive $e^{\text{th}}$ root of unity. Let $\mgb$ be the moduli functor on the category of schemes over $\Z[\frac{1}{e}, \zeta_e]$ that classifies admissible stable $G$-covers of $\Po$, and denote by $\mG$  the smooth locus of $\mgb$. Both $\mG$ and $\mgb$ have good reduction modulo $p$. Within each irreducible component of $\mG$, the monodromy datum of the parameterized curves is constant.  Conversely, given a monodromy datum $(G,r,\au)$, the substack $\mgra$ of $\mG$ parametrizing $G$-covers with monodromy $(G,r,\au)$ is irreducible.

\subsection{The Shimura variety $\shgf$}\label{Shimura}  We briefly recall the construction of the PEL type moduli space $\shgf$. We refer to \cite[Section 3.2, 3.3]{Second Paper} for more details.

Let $\mathcal{A}_g$ denote the moduli space of principally polarized abelian varieties of dimension $g$, and $T:\mathcal{M}_g\to \mathcal{A}_g$ the Torelli morphism. 
For $\gra$ a monodromy datum as in Definition \ref{monodromy datum},  we denote by $S\gra$ the largest closed, reduced and irreducible substack of $\mathcal{A}_g$ containing $T(\mgra)$ such that the action of $\Z[G]$ on the Jacobian  of the
universal family of curves over $\mgra$ extends to the universal abelian scheme over $S\gra$. By \cite{delignmostow},  $S\gra$  is an irreducible component of a PEL type moduli space, which we denote by $\shgf$  (see also \cite[Section 3]{Second Paper}).  The moduli space $\shgf$ has a canonical model over $\Z[\frac{1}{e}, \zeta_e]$, and good reduction at all primes $p\nmid |G|$.

We recall the definition of the PEL datum associated with $\shgf$. 
Let $\Q[G]$ denote the $\Q$-group algebra of $G$,  and $*$ the involution on $\Q[G]$ induced from the group homomoprhism $g\mapsto g^{-1}$.  Fix $\xb \in \mgra(\C)$, and denote by $C=\ct$  the associated curve over $\C$. Let  $V=H^1(C,\Q)$ be the first Betti cohomology group of $C$; the action of $G$ on $C$ induces a structure of $\Q[G]$-vector space on $V$. We denote by $\langle\cdot,\cdot\rangle$ the standard skew Hermitian form on $V$, and by $h$ the Hodge structure on $V$. 
The Shimura datum of $\shgf$ is defined by the PEL-datum $(\Q[G], *, V, \langle\cdot, \cdot\rangle, h)$, and is independent of the choice of $\xb$. 

We denote 
 by $\fb$  the {\em signature} of the multiplication by $\Q[G]$ on $V$, that is, the signature of  $\mathcal{G}=GU(V, \langle\cdot, \cdot\rangle )$ the group of $\Q[G]$-linear similitudes of $V$. 
Concretely, $\fb$ is defined as follows. Denote the Hodge structure $h$ of $V$ by $V\otimes_\Q\C=V^+\oplus V^-$, where $V^+=H^0(C,\Omega^1)$, via the Betti-de Rham comparison isomorphism.
Let $\tg$ be the group of characters of $G$, $\tg=\Hom(G,\C^*)$. We define $$\fb:\tg\to \Z \text{ as }f(\tau)=\dim( {V}^{+}_{\tau}),$$ where for each $\tau\in \tg$ we denote by $V_\tau^+$ the subspace of $V^+$ of weight $\tau$. 
The involution on $\Q[G]$ induces an involution on $\tg$, where for $\tau \in \tg$, $\tau^*(g)=\tau(g^*)$.
By definition, $f(\tau^*)=\dim V^-_\tau$, and for each $\tau\in \tg$, $\tau\neq \tau^*$, the pair $(f(\tau),f(\tau^*))$ is the signature of the unitary group $GU(V, \langle\cdot, \cdot\rangle )$ at the real place underlying $\tau$ and $\tau^*$.
The signature $\fb$ can be computed explicitly from the monodromy datum $\gra$ via the Hurwitz--Chevalley-Weil formula (see \cite[Theorem 2.10]{Chevalley}).
For $G$ a cyclic group of size $m$, after identifying $\tg\simeq\{0, 1, \dots m-1\}$, we have (see \cite[Lemma 2.7, Section 3.2]{special family})\begin{align}\label{signature}
    f(\tau_i)=-1+\sum_{k=1}^r \langle\frac{-ia_k}{m}\rangle \text{ for } 1\leq i\leq m-1, \text{ and } f(\tau_0)=0.
\end{align}

With abuse of notation, in the following we denote by $(G,\fb)$ the Shimura datum of $\shgf$.

 \subsubsection{The $(G,\underline{f})$-ordinary stratum at unramified primes}

Let $p$ be a prime not dividing $|G|$. Then $p$ is a prime of good reduction for $\shgf$ and by \cite{viehmann-wedhorn} both the Ekedahl--Oort and Newton stratification of $\shgf_{\fpb}$ are well understood. We briefly recall some of their properties.

The Newton polygon is a discrete invariant that classifies the isogeny class of the $p$-divisible group of a polarized abelian variety over $\fpb$, and is known to induce a stratification on $\mathcal{A}_{g,\fpb}$. By \cite{viehmann-wedhorn}, the Newton polygons corresponding to non-empty strata in $\shgf_{\fpb}$ are in one-to-one correspondent with the elements in the associated Kottwitz set at $p$, its natural partial order agreeing with specialization on $\shgf_{\fpb}$. In \cite{Kottwitz}, this set is denoted by $B(\mathcal{G}_{\Q_p},\mu_h)$, where $\mathcal{G}=GU(V, \langle\cdot, \cdot\rangle )$
and $\mu_h$ is the $p$-adic cocharacter induced by the Hodge structure $h$.
By \cite{Rapoport-RIcharts} and \cite{Wedhorn}, there is a unique maximal element / lowest polygon in $B(\mathcal{G}_{\Q_p},\mu_h)$, corresponding to the unique open (and dense) Newton stratum in  $\shgf_{\fpb}$; this is known as the $\mu$-ordinary polygon at $p$ and in our context can  be computed explicitly from the splitting behaviour of p in the group algebra $\Q[G]$ and the signature $\fb$ (for example, it is ordinary if $p$ is totally split in  $\Q[G]$).  

The Ekedahl-Oort type is a discrete invariant that classifies the isomorphism class of the $p$-kernel of a polarized abelian variety over $\fpb$, and also induces a stratification on  $\mathcal{A}_{g,\fpb}$. By \cite{viehmann-wedhorn}, the Ekedahl--Oort types corresponding to non-empty strata in $\shgf_{\fpb}$ are in one-to-one correspondence with certain elements in the Weyl group of the reductive group $\mathcal{G}$, their dimension equal to the length of the element in the Weyl group.
In particular, there is a unique element of maximal length, corresponding to the unique non-empty open (and dense) Ekedahl--Oort stratum in  $\shgf_{\fpb}$. The Ekedahl--Oort type corresponding to the maximal element is called $p$-ordinary. 

By \cite[Theorem 1.3.7]{Moonen EO type formula}, the $p$-ordinary Ekedahl--Oort stratum and
$\mu$-ordinary Newton stratum of $\shgf_{\fpb}$ agree.
As their definition depends on the Shimura datum $(G,\underline{f})$ and the prime $p$, we refer to it as the $(G,\underline{f})$-ordinary stratum at $p$, and denote the associated Newton polygon by $\mu_p(G,\fb)$.  
An explicit formula for the polygon $\mu_p(G,\fb)$ is given   \cite[Proposition 4.3]{Second Paper}, as a special case of that in \cite[Section 1.2.5]{Moonen EO type formula}. 
We briefly recall some aspects of its construction.

 \subsubsection{The $(G,\underline{f})$-ordinary polygon}
Given the rational prime $p$, we fix an algebraic closure $\overline{\Q}_p$ of $\Q_p$ and an isomorphism $\iota:\Hat{\overline{\Q}}_p\simeq \C$.
We denote by $\overline{\Q}^{\rm un}_p$ the maximal unramified subfield of $\overline{\Q}_p$, and by $\fpb$ its residue field. 
Since $p\nmid |G|$, $\iota$ induces an isomorphism $\tg\simeq \Hom(G,\overline{\Q}^{\rm un*}_{p})$. 
Let $\sigp$ be the Frobenius element in $\gal(\fpb/\mathbb{F}_p)$, then $\sigp$ lifts to an element of $\gal(\overline{\Q}^{\rm un}_p/\Q_p)$, and we consider the action of $\sigp$ on $\tg$ by composition, that is $\tau^{\sigp}(x)=\sigp(\tau(x))=\tau(x)^p$.
This action partitions $\tg$ into Frobenius orbits, and we denote the set of Frobenius orbits of $\tg$ by $\og$. For $\tau \in \tg$, we use $\orho$ to denote the Frobenius orbit of $\tau$. 
The Frobenius orbits in $\og$ are naturally in one-to-one correspondence with the simple factors of $\Q_p[G]$.  From the decomposition into simple factors of $\Q[G]$, $\Q[G] \cong \prod_H K_H $ where $H$ varies among the subgroup of $G$ such that $G/H$ is cyclic, we deduce
\begin{equation}\label{simplefactors}
    \begin{split}
       \Q_p[G] &\cong \prod_{\substack{H \leq G \\ G/H \text{ cyclic }}}K_H \otimes_{\Q}\Q_p \cong  \prod_{\substack{H \leq G \\ G/H \text{ cyclic }}}\prod_{\substack{\orho\\
       \ker(\tau)=H }}K_{\orho},\\
    \end{split}
\end{equation}
where each Frobenius orbit $\OO$ corresponds to a prime $\p$ above $p$ in $K_H$, for $H=\ker(\tau)$,  and $K_{\OO}$ is the completion of $K_H$ at this prime.

Let $C \to \Po$ be an abelian cover parameterized by a point $\xb \in \mgra(\fpb)$, we denote its Jacobian by $J(C)$. Let ${D}={D}(J(C))$ denote the Diedonn\'e module of the abelian variety $J(C)/\fpb$, and $\text{NP}({D})$ its Newton Polygon at $p$.
Then, the structure of $\Z_p[G]$-module on  $D$ induces a decomposition up to isogeny  $D \sim \bigoplus_{\OO \in \og}D_{\OO}$, and hence an equality of Newton polygon  $\text{NP}(D)=\bigoplus_{\OO \in \og}\text{NP}(D_{\OO})$.

From the formula given in \cite[Proposition 4.3]{Second Paper}, the Newton polygon $\mu_p(G, \fb)$ also decomposes as $\mu_p (G,\fb)=\bigoplus_{\OO \in \og}\mu(\OO)$, where for each orbit $\OO$ the polygon $\mu(\OO)$ only depends on the values $(f(\tau))_{\tau \in \OO}$. Furthermore, $\text{NP}(D)= \mu_p(G, \fb)$ if and only if $\text{NP}(D_{\OO})=\mu(\OO)$, for each Frobenius orbit $\OO$.

\subsection{Hasse--Witt triples and Ekedahl--Oort types}\label{dieu}
Recall $\sigma$ denotes the Frobenius of $\fpb$.
Let $A$ be a principally polarized abelian variety of dimension $g$, defined over $\fpb$. Its Ekedahl-Oort type encodes the isomorphism class of $A[p]$, or equivalently the isomorphism class of the associated polarized mod-$p$ Diedonn\'e module $(M,F,V,b)$, where
\begin{itemize}
    \item $M=H^1_{dR}(A/\fpb)$;
    \item $F: M \to M$ is the $\sigma$-linear map on $M$ induced by the Frobenius of $A$;
    \item $b: M \times M \to \fpb$ is the pairing induced by the polarization of $A$;
    \item $V: M \to M$ is the unique $\sigma^{-1}$-linear operator satisfying $b(F(x),y)=b(x,V(y))^{p}$.
\end{itemize}

In \cite{Moonen algorithm}, Moonen establishes a equivalence of category between the polarized mod-$p$ Dieudonné modules and Hasse-Witt triples, where he defines a Hasse-Witt triple $(Q,\phi,\psi)$ as follows:
\begin{itemize}
    \item $Q$ is a finite dimensional vector space over $\fpb$;
    \item $\phi: Q \to Q$ is a $\sigma$-linear map;
    \item $\psi: \ker(\phi) \to \im(\phi)^{\perp}$ is a $\sigma$-linear isomorphism, where $\im(\phi)^{\perp}\subseteq Q^{\smvee} ={\rm Hom}_{\fpb} (Q,\fpb)$ is the subspace $\im(\phi)^{\perp}=\{\lambda \in Q^{\smvee}: \lambda(\phi(q))=0, \forall q \in Q$\}.
\end{itemize}
Under Moonen's equivalence of category,  the polarized mod-$p$ Dieudonné module $(F,M,V,b)$ corresponding to a Hasse-Witt triple $(Q,\phi,\psi)$ is given by: 
\begin{itemize}
    \item $M = Q \oplus Q^{\smvee}$;
    \item $F: M \to M$ is defined as follows: set $R_1=\ker(\phi)$, choose $R_0$ a compliment of $R_1$ in $Q$, and write $M=(R_0\oplus R_1)\oplus Q^{\smvee}$, then $F(x+y,z)=(\phi(x),\psi(y))$, for any $x\in R_0,$ $y\in R_1$, $z\in Q^{\smvee}$;
    \item $b:M\times M\to \fpb$ is defined by $b((q,\lambda),(q',\lambda'))= \lambda'(q)-\lambda(q')$, for any $q,q'\in Q$ and $\lambda,\lambda'\in Q^{\smvee}$.
    \item  $V:M\to M$ is uniquely determined by $b(F(x),y)=b(x,V(y))^p$.
\end{itemize}

\subsubsection{The $(G,\fb)$-ordinary Ekedahl--Oort type}
We recall the definition of the $p$-ordinary Ekedahl--Oort type for $\shgf$. Since it depends on the Shimura datum $(G,\fb)$ we also refer to it as the $(G,\fb)$-ordinary Ekedahl--Oort type at $p$.

Recall the identification $\tg=\Hom(G,\C^*) \cong \Hom (G,\overline{\Q}_p^{\rm un*})$; it induces an isomorphism $\tg\cong \Hom(G,\fpb^*)$.
Let $A$ be an abelian variety  over $\fpb$ corresponding to a point of  $\shgf$, and denote its mod-$p$ Dieudonné module by $(M, F,V,b)$. The action of $G$ on $A$ induces a structure of $\fpb[G]$-module on $M$. Hence, the Dieudonné module $M$ decomposes as 
$$M=\bigoplus_{\tau\in\tg} M_\tau=\bigoplus_{\OO\in\og}M_{\OO},\text{ where } M_{\OO}=\bigoplus_{\tau\in \OO}M_{\tau}, $$
and  for each $\tau \in\tg$,  $M_\tau$ is the $\tau$-isotypic component of $M$. That is, $h \in G$ acts on $M_\tau$ via multiplication by $\tau(h)$. For $\tau\in \tg$, let $g(\tau)=\dim_{\fpb} (M_{\tau})$. Since it depends only on the Frobenius orbit of $\tau$, we write  $g(\OO)=g(\tau)$, for any/all $\tau \in \OO$.

For simplicity, given $\tau\in\tg$, we denote $\tau^{\sigp}$ by $p\tau$ and its orbit $\OO_{\tau}=\{\tau,p\tau,\dots,p^{|\OO_{\tau}|-1}\tau\}$. Since $p(p^{|\OO_\tau|-1}\tau)=\tau$, we also write $p^{|\OO_\tau|-1}\tau$ as $\frac{\tau}{p}$.  
Then, $F$ maps $M_{\tau}$ to $M_{p\tau}$, and $V$ maps $M_{\tau}$ to $M_{\frac{\tau}{p}}$. 

Recall, for $\tau\in\tg$, $\tau^*\in\tg$ is defined as $\tau^*(x)=\tau(x)^{-1}$. Given an orbit $\OO$, we denote its conjugate orbit as ${\OO}^*=\{\tau^*\mid \tau\in \OO\}$. The polarization $b:M \times M \to \fpb$ identifies $M^{\smvee}$ with $M$, $M_\OO^{\smvee}$ with $M_{\OO^*}$, and $M_{\tau}^{\smvee}$ with $M_{\tau^*}$.

\begin{definition}\label{Def_EOMFVb} (\cite[Section 1.2.3]{Moonen EO type formula})
The $(G,f)$-ordinary mod-$p$ Diedonn\'e module $(M,F,V,b)$ is given as follows. 
Let $\{e_{\tau,j}\mid 1\leq j\leq {g(\OO_\tau)}\}$ be a $\fpb$-basis of $M_\tau$, and denote the dual basis on $M_\tau^{\smvee}$ by $\{\check{e}_{\tau,j}\mid 1\leq j\leq {g(\OO_\tau)}\}$. 
Then the polarization $b$ on $M$, in terms of the induced isomorphisms $M_\tau^{\smvee}\simeq M_{\tau^*}$  for $\tau\in\tg$, is given by $\check{e}_{\tau,j}\mapsto e_{\tau^*,g(\OO)+1-j}$, for $1\leq j\leq g(\OO_\tau)$.
The action of $F$ and $V$ on $M$, when restricted to $M_\tau$ for $\tau\in\tg$, are given by
\begin{align}\label{Eq_MFVb}
F(e_{\tau,j})&=\begin{cases}
    e_{p\tau,j} &\text{ if } j \leq \fts\\
    0 &\text{ if } j \geq \fts+1,
    \end{cases}
&
V(e_{p\tau,j_1})&=\begin{cases}
    0 &\text{ if } j_1 \leq f(\tau^*)\\
    e_{\tau,j_1} &\text{ if } j_1 \geq f(\tau^*)+1.
    \end{cases}
\end{align}

\end{definition}

\begin{remark}\label{Rem_EOHassetrople}
Under Moonen's equivalence, the $(G,\fb)$-ordinary Hasse--Witt triple $(Q, \phi,\psi)$ is defined as follows.
Let $Q=\ker(F)^{\smvee}\subseteq M$ and define $Q_{\tau}=Q\cap M_{\tau}$, for each $\tau\in\tg$. Write $Q_{\tau^*}^{\smvee}=(Q_{\tau^*})^{\smvee}$ (in general $Q_{\tau^*}^{\smvee}$ is not $Q^{\smvee}\cap M_{\tau^*}$).
Then $M_{\tau}=Q_{\tau} \oplus Q_{\tau^*}^{\smvee}$,
where the set $\{e_{\tau,i_{\tau,1}},\dots, e_{\tau,i_{\tau,f(\ts)}}\}$ is a basis of $Q_{\tau}$ and $\{e_{\tau,j_{\tau,1}},\dots, e_{\tau,j_{\tau,f(\tau)}}\}$ is a basis of $Q_{\tau^*}^{\smvee}$.
With respect to this choice of bases for $\qt,\qtsd$, for all $\tau\in\tg$, the matrix of $F$ restricted to $M_\tau$, that is $F: M_\tau=Q_{\tau} \oplus Q_{\tau^*}^{\smvee}\to M_{p\tau}=\qpt \oplus \qptsd$ is
\begin{align}\label{Eqn: Ftau V'tau}
F_{\tau}= \begin{bmatrix}
    \phi_\tau & 0\\
    \psi_\tau & 0 \\
\end{bmatrix}
\end{align}
where $\phi_\tau$ (respectively $\psi_\tau$) is the matrix of $\phi$ (respectively $\psi$) restricted to $Q_{\tau}$, that is $\phi_\tau:Q_\tau\to\qpt$ (respectively $\psi_\tau:\qt \to \qptsd$).
\end{remark}

\subsubsection{Elements in the Weyl group}

\newcommand{\EO}{Ekedahl--Oort }

\newcommand{\Sym}{{{\rm Sym}}}
By \cite{viehmann-wedhorn}, the \EO types associated with non-empty strata of $\shgf_{\fpb}$ are in one-to-one correspondence with certain elements in the Weyl group of the reductive group $\mathcal{G}$. We recall this construction.

Consider the set $${\rm Weyl}(G,f)=\prod_{\tau \in \tg}\Sym_{g(\OO)}/W_{f,\tau},$$
where, for each $\tau\in\tg$, $W_{f,\tau}=\Sym\{1, \dots,f(\tau)\} \times \Sym\{f(\tau)+1, \dots, g(\OO)\}.$
Then, the non-empty \EO strata of $\shgf_{\fpb}$ are in one-to-one correspondence with the cosets in ${\rm Weyl}(G,f)$ defined by elements
$w=(\wt\mid \tau\in\tg)\in \prod_{\tau \in \tg} \Sym_{g(\OO_\tau)}$
satisfying $$w_{\tau^*}(j)=g(\OO)+1-w_{\tau} (g(\OO)+1-j).$$ In particular, the open Ekedahl–Oort stratum corresponds to the unique element $w$ such that the permutations $\wt$ have maximum length, for all $\tau\in\tg$.

The \EO type of $A$ is defined in terms of the canonical filtration of $M$, of length $2g=\dim(M)$,  obtained by repeatedly applying $F$ and $V^{-1}$ to $M$. By projecting the filtration to $M_{\tau}$, we obtain a filtration of $M_\tau$, of length $g(\OO_\tau)=\dim(M_{\tau})$,
$0\subsetneq M_{\tau,1} \subsetneq M_{\tau,2} \subsetneq \dots, \subsetneq M_{\tau,g(\OO)}=M_{\tau}.$
To each $\tau$, we associate a permutation $w_{\tau}\in \Sym_{g(\OO)}$ as follows. 
For $1\leq j\leq g(\OO)$, denote $\eta_{\tau,j}=\dim(\ker(F)\cap M_{\tau,j})$. Note that $\eta_{\tau,j}\leq \eta_{\tau,j+1}\leq \eta_{\tau,j}+1$. 
 Recall that by definition, the signature $\fb$ satisfies $f(\tau)=\dim(\ker(F)\cap M_{\tau})$ and $f(\tau)+f(\tau^*)=g(\OO)$.
We deduce that $0\leq \eta_{\tau,j}\leq f(\tau)$. We record the $k^{th}$ position at which the sequence of $\eta_{\tau,j}$ jumps by $j_{\tau,k}$. We obtain 
\begin{align}\label{Def_j}
1\leq j_{\tau,1}<j_{\tau,2}<\dots<j_{\tau,f(\tau)}\leq g(\OO)\end{align}
satisfying $\eta_{\tau,j_{\tau,k}}=\eta_{\tau,j_{\tau,k}-1}+1$. We denote by $i_{\tau,1}<\dots<i_{\tau,f(\tau^*)}$ the remaining indices, they satisfy  $\eta_{\tau,i_{\tau,k}}=\eta_{\tau,i_{\tau,k}-1}$. 

\begin{definition}\label{Def_EOword}
The \EO type of $A$ is the coset in ${\rm Weyl}(G,f)$ of the element $w=(w_\tau\mid \tau\in\tg)$ where  $w_{\tau} \in \Sym_{g(\OO_\tau)}$is given by
$$w_{\tau}(j_{\tau,k})=k \text{ and } w_{\tau}(i_{\tau,k})=f(\tau)+k.$$
\end{definition}
By definition, $w_\tau\in \Sym_{g(\OO_\tau)}$ satisfies the property 
\begin{equation}\label{wordproperty}
 \text{ if } j'< j \text{ and } \wt(j') > \wt(j) \text{ then }\wt(j) \leq f(\tau) < \wt(j')
\end{equation}
Furthermore, $\wt$ is the unique element in its coset in $S_{g(\OO)}/W_{f,\tau}$ that satisfies (\ref{wordproperty}).

\begin{lemma}\cite[Section 2.3.4]{Moonen dimension formula}\label{length}
The permutation $w_{\tau}\in \Sym_{g(\OO)}$ has length
$$\sum_{k=1}^{f(\tau)}j_{\tau,k}-w_{\tau}(j_{\tau,k})=\sum_{k=1}^{f(\tau^*)} w_{\tau}(i_{\tau,k})-i_{\tau,k}.$$
Moreover, this quantity is maximized if and only if $\ker(F)\cap M_{\tau,f(\tau^*)}=\{0\}$.
\end{lemma}
\begin{remark}\label{rmk}
The condition $\ker(F)\cap M_{\tau,f(\tau^*)}=\{0\}$ is equivalent to the equalities $j_{\tau,k}=\fts+k$ for $1\leq k\leq f(\tau)$,  and  $i_{\tau,k}=k$ for $1\leq k\leq \fts$. We deduce that $w_{\tau}$ has maximal length if and only if $w_{\tau^*}$ has maximal length, if and only if  
\[\wt(k)=\begin{cases} f(\tau)+k \text{ for } k\leq \fts,\\   k-\fts \text{ for } k>\fts.\end{cases} \]
\end{remark}

\section{Reduction from abelian cover to cyclic covers}\label{reduction argument}

In this section, we reduce the proof of Theorem \ref{abelian cover} to the case when $G$ is a cyclic group. More precisely, we show that an abelian $G$-cover of $\Po$ is $(G,\fb)$-ordinary if its cyclic quotients are.

Recall the identification $\tg= \Hom(G,\C^*) =\Hom(G,\overline{\Q}^*)\simeq \Hom(G,\fpb^*)$, and consider the action of $\galq$ on $\tg$ by composition on the left.
For any subgroup $H \leq G$, we denote $$\tg^H= \{ \tau \in \tg | H \subseteq \ker(\tau)\}\text{ and }\tg^{H,\text{new}}=\{\tau \in \tg | H=\ker(\tau)\}.$$ For $H=\{1\}$, we also write $\tgun=\tg^{\{1\},\text{new}}$. 
Consider the partition 
$$\tg=\bigcup_{\substack{H \leq G,\\ G/H \text{cyclic}}}\tguh=\coprod_{\substack{H \leq G,\\ G/H \text{cyclic}}}\tghn$$
Then the action of $\galq$ on $\tg$ preserves the partition, and for each $H$, with $G/H$ cyclic,  $\galq$ acts transitively on $\tghn$.

Recall the decomposition of $\Q[G]$ into the simple factors,
$\Q[G] \cong \prod K_H,$ where $H$ varies among the subgroups $H$ for which $G/H$ is cyclic.
We denote the induced decomposition of $J(C)$ up to isogeny as
$$J(C) \sim \bigoplus_{\substack{H \leq G,\\ G/H \text{ cyclic } }} J(C)_{H}.$$
\begin{definition}\label{definition of new part}
When $G$ a cyclic group, we call $J(C)^{\n}=J(C)_{\{1\}}$ the new part of the Jacobian.
\end{definition}

By construction, for any subgroup $H$ of $G$, with $G/H$ cyclic, after identifying $\mathcal{T}_{G/H}\simeq \tg^H$, we have
$$J(C/H)\simeq \bigoplus_{H\leq H'} J(C)_{H'}.$$
where the signature of the $G/H$-cover $C/H$ is $\fb_{G/H}=\fb_{\vert \tg^H}$ and the signature of the new part $J(C/H)^{\n}$ of $J(C/H)$ is $\fb_{G/H}^{\n}=\fb_{\vert T^{H,\n}_G}$.

From the formula computing  $\mu$-ordinary polygons, and the decomposition $\mu_p(G,\fb)=\bigoplus_{\OO\in\OO_G} \mu(\OO)$ arising from (\ref{simplefactors}), we deduce 
$$\mu_p(G,\fb)= \bigoplus_{\substack{H \leq G,\\ G/H \text{ cyclic } }} \mu_p(K_H,\fb_{G/H}^{\n}) \text{ where }
\mu_p(K_H,\fb_{G/H}^{\n})= \bigoplus_{\substack{\OO \in \OO_{G}\\ \OO\subseteq \mathcal{T}_{G/H}^{\n}}}\mu(\OO)$$
and 
$$\mu_p(G/H, \fb_{G/H})=\bigoplus_{\substack{\OO \in \OO_{G/H}}} \mu(\OO) \text{ where } \OO_{G/H}\simeq \{\OO\in\OO_G\mid \OO\subseteq \mathcal{T}_G^H\}.$$ 

In the following, we refer to the Newton polygon $\mu_p(K_H,\fb_{G/H}^{\n})$ as the $(K_H,\fb_{G/H}^{\n})$-ordinary polygon at $p$. By definition, it is the $\mu$-ordinary polygon at $p$ of a PEL type Shimura variety parametrizing abelian varieties with an action of the field $K_H$, and signature $\fb_{G/H}^{\n}$.

We deduce the following statement.

\begin{lemma}\label{equivalence two}
Let $G$ be an abelian group, and $p$ a prime $p\nmid |G|$. For $C \to \Po$ a $G$-cover of $\Po$ defined over $\fpb$,
the following are equivalent:
\begin{enumerate}
    \item $J(C)$ is $(G,\fb)$-ordinary;
    \item $J(C/H)$ is $(G/H,\fb_{G/H})$-ordinary, for all $H \leq G$ with $G/H$ cyclic; 
    \item $J(C/H)^{\n}$ is $(K_H, \fb_{G/H}^{\n})$-ordinary, for all $H \leq G$ with $G/H$ cyclic.
\end{enumerate}
\end{lemma}

Since $\mu$-ordinariness is an open condition, Lemma \ref{equivalence two} reduces the proof of Theorem \ref{abelian cover} to the case of $G$ a cyclic group. 
Furthermore, by the remark below, we may assume that $G$ is cyclic of size $l\geq 3$.

\begin{remark}\label{R_assl3}
 Let $C$ is a cyclic cover of $\Po$ of degree $2$, branched at $r$ points. Assume $r\leq 5$.  Since the number of branched points of a cover of degree $2$ is even, we deduce that $r\leq 4$. If $r=2$, the genus of $C$ is $0$.
 If $r=4$, the genus of $C$ is $1$, and $C$ is generically ordinary.
\end{remark}

\begin{example}\label{Moonen Oort}
We cite an example from \cite{Moonen Oort} to illustrate how to compute the quotient covers and their signatures.
Consider the abelian monodromy datum $(\Z/2\Z \times \Z/6\Z, 4, (1,0),(1,1),(0,2),(0,3))$. Then $\tg$ has size 12; for $0\leq i\leq 1$ and $0\leq j\leq 5$, we denote by
 $\kij\in\tg$ the character given by  $\chi_{i,j}(a_1,a_2)=\zeta_6^{3a_1i+a_2j}$, for
 $\zeta_6=e^{\pi i/3}\in \C$ is a primitive sixth root of unity.
Then $\tg$ is  partitioned into the following $8$ Galois orbits:
\begin{equation*}
(\chi_{0,0}),(\chi_{0,1},\chi_{0,5}),
(\chi_{0,2},\chi_{0,4}),(\chi_{0,3}),
(\chi_{1,0}),(\chi_{1,1},\chi_{1,5}),
(\chi_{1,2},\chi_{1,4}),(\chi_{1,3}), 
\end{equation*}
Denote $C_{\kij}=C/\ker(\kij)$. The $8$ Galois orbits lead to $8$ quotient curves, we illustrate one of them.
    Consider the Galois orbit $(\chi_{1,2},\chi_{1,4})$. Let $\rho=\chi_{1,2}$, $H=\ker(\rho)=\{(0,0),(0,3)\}$. The monodromy datum of $C_{\rho}$ is $(6,3,(3,5,4))$. This is because $|\rho(G)|=6$ and $\rho(\au_i)=\zeta_6^{\ub_i}$ where $\ub=(3,5,4,0)$. Thus, $C_\rho$ is the normalization of $y^6=(x-x_1)^3(x-x_2)^5(x-x_3)^4$.
    The characters whose kernel contains $\ker(\rho)$ are $\{\chi_{0,0},\chi_{1,2},\chi_{0,4},\chi_{1,0},\chi_{0,2},\chi_{1,4}\}$, and they arise by inflation from the characters of $G/H$. We can compute their signature via \autoref{signature}, for $1\leq i\leq 5$
    $$f_G(\chi_{i,2i})=f_{G/H}(i^{-1}(\chi_{i,2i}))=f_{G/H}(\rho^i)=\langle\frac{-3i}{6}\rangle+\langle\frac{-5i}{6}\rangle+\langle\frac{-4i}{6}\rangle-1.$$
    Hence, the signature $\fb_G$ on $\{\chi_{0,0},\chi_{1,2},\chi_{0,4},\chi_{1,0},\chi_{0,2},\chi_{1,4}\}$ takes values $\{0, 0,0,0,0,1\}$. We deduce that $J(C_\rho)$ is an elliptic curve, with complex multiplication by $K_6=\Q(\zeta_6)$.
\end{example}

\section{A criterion of $p$-ordinariness for extended Hasse-Witt matrices}\label{combi}
Let $(G,r,\au)$ be an abelian monodromy datum as in Definition \ref{monodromy datum}, and consider the associated Shimura datum $(G,\fb)$ as defined in Section \ref{Shimura}. Assume $p$ is a prime not dividing $|G|$. In this section, we give an explicit numerical criterion (Theorem \ref{numerical criteria}) for the mod-$p$ Dieudonn\'e module of abelian covers of $\Po$ with monodromy $(G,r,\au)$ to be $(G,\fb)$-ordinary. We deduce the statement from \cite[Theorem 1.3.7]{Moonen EO type formula}, which when specialized to our context states that a mod-$p$ Diedonn\'e module is $p$-ordinary if, for each $\tau\in\tg$, the associated word $\wt$ has maximal length. Our criterion is stated as finitely many rank conditions on iterations of  the extended Hasse-Witt matrix.

In the following,  $\OO\subseteq \mathcal{T}_G$ is a Frobenius orbit, and we assume $\tau\in\OO$.

With the notations from Section \ref{dieu}, given an \EO type $w=(w_\tau\mid \tau\in \mathcal{T}_G)$ as in Definition \ref{Def_EOword},  we describe the associated mod-$p$ Diedonn\'e module $(M,F,V,b)$ with a $(G,\fb)$-structure (that is, an action of $G$ of signature $\fb$). 
Let ${e_{\tau,1},\dots, e_{\tau,g(\OO)}}$ be a basis of $M_{\tau}$, and $f(\tau)=\dim(\ker(F)\cap M_{\tau})$.
For suitable
increasing sequences $1\leq j_{\tau,1}<j_{\tau,2}<\dots<j_{\tau,f(\tau)}\leq g(\OO)$ and $1\leq i_{\tau,1}<\dots<i_{\tau,f(\tau^*)}\leq g(\OO)$ satisfying 
$\{j_{\tau,1},\dots,j_{\tau,f(\tau)}\}\cup\{,i_{\tau,1},\dots,i_{\tau,f(\tau^*)}\}=\{1,\dots,g(\OO)\}$, the word $w_{\tau}$ associated to $(M,F,V,b)$ is given by $$w_{\tau}(j_{\tau,k})=k \text{ and }w_{\tau}(i_{\tau,k})=f(\tau)+k.$$ 
From \autoref{Eq_MFVb}, the action of $F$ and $V$ on $M_{\tau}$ is determined by  $w=(\wt)_{\tau \in \tg}$ as 
\begin{align}\label{Eq_F}
F(e_{\tau,j})&=\begin{cases}
    e_{p\tau,w_{\tau}(j)-f(\tau)} &\text{ if } w_{\tau}(j) \geq f(\tau)+1\\
    0 &\text{ if } w_{\tau}(j) \leq f(\tau)
    \end{cases},
&
V(e_{p\tau,j_1})&=\begin{cases}
    0 &\text{ if } j_1 \leq f(\tau^*)\\
    e_{\tau,w_{\tau}^{-1}(j_1-f(\tau^*))} &\text{ if } j_1 \geq f(\tau^*)+1.
    \end{cases}
\end{align}

\begin{lemma}\label{Lem: 2 technical lemma w tau}
Given $\tau\in\OO$ and $1\leq k\leq g(\OO)$, there exists $k' \in \mathbb{N}$ such that $1\leq k'\leq g(\OO)$ and
$$\{w_{\tau}(j)-f(\tau)\mid 1\leq j\leq k, w_{\tau}(j)> f(\tau)\} =\{1,2,\dots,k'\}.$$
Further, we have $k'\leq \min(k,f(\tau^*))$.
Moreover, if $w_{\tau}$ is maximal then we have $k'=\min(k,f(\tau^*))$.
\end{lemma}
\begin{proof}

Firstly, note that $\{j\mid w_{\tau}(j)> f(\tau)\}=\{i_{\tau,1},\dots,i_{\tau,f(\tau^*)}\}$.
Let $k'=\max\{a \mid 1\leq a\leq f(\tau^*),  i_{\tau,a}\leq k\}$. Set $k'=0$ if $k<i_{\tau,1}$. We see that $\{j\mid 1\leq j\leq k, w_{\tau}(j)> f(\tau)\}=\{i_{\tau,1},\dots,i_{\tau,k'}\}$.
Therefore,
\[\{w_{\tau}(j)-f(\tau)\mid 1\leq j\leq k, w_{\tau}(j)> f(\tau)\} =\{1,2,\dots,k'\}.\]
It is clear that $k'\leq f(\tau^*)$.
Further, since $a\leq i_{\tau,a}$, we have $k'\leq i_{\tau,k'}\leq k$. This shows that $k'\leq \min(k,f(\tau^*))$.
If $w_{\tau}$ is maximal, then we have $i_{\tau,a}=a$. Therefore, if $f(\tau^*)\leq k$, then we have $k'=f(\tau^*)$. On the other hand, if $f(\tau^*)\geq k$, then we have $k'=k$.
We see that in both cases we have $k'=\min(k,f(\tau^*))$.
\end{proof}
We deduce the following lemma. Recall, for $1\leq k\leq g(\OO)$,
$M_{\tau,k}$ denotes the subspace of $M_\tau$ spanned by $e_{\tau,1},\dots, e_{\tau,k}$.
\begin{lemma}\label{F}
Given $\tau\in\OO$ and $1\leq k\leq g(\OO)$, let $l=\dim(F(M_{\tau,k}))$. Then $F(M_{\tau,k})=M_{p\tau,l}$ and $l\leq \min(k,f(\tau^*))$. Moreover, if $w_{\tau}$ is maximal, then $l= \min(k,f(\tau^*))$.
\end{lemma}
\begin{proof}
We know that $F(M_{\tau,k})$ is spanned by $\{F(e_{\tau,1}),\dots,F(e_{\tau,k})\}$. Therefore, it is spanned by
$$\{e_{p\tau,w_{\tau}(j)-f(\tau)}\mid 1\leq j\leq k, w_{\tau}(j)\geq f(\tau)+1\}.$$
This set is also linearly independent and hence forms a basis of $F(M_{\tau,k})$. Consider the $k'$ given by Lemma \ref{Lem: 2 technical lemma w tau}, we see that $F(M_{\tau,k})=M_{p\tau,k'}$. Since $l=\dim(F(M_{\tau,k}))$, we see that $l=k'$. Therefore, Lemma \ref{Lem: 2 technical lemma w tau} tells us that $l\leq \min(k,f(\tau^*))$ and if $w_{\tau}$ is maximal then $l=\min(k,f(\tau^*))$.
\end{proof}

 From Remark \ref{Rem_EOHassetrople}, recall $M_{\tau}=Q_{\tau}\oplus Q_{\ts}^{\smvee}$, where $Q_\tau$ is spanned by $\{e_{\tau,i_{\tau,1}},\dots, e_{\tau,i_{\tau,f(\ts)}}\}$ and $Q_{\tau^*}^{\smvee}$ by 
$\{e_{\tau,j_{\tau,1}},\dots, e_{\tau,j_{\tau,f(\tau)}}\}$; $\pi_{\tau}:M_{\tau}\to Q_{\tau}$ denotes the canonical projection, with kernel $Q_{\tau^*}^{\smvee}$.

For any $\tau\in \OO$, $1\leq j\leq g(\OO)$, and Frobenius orbit $\OO$, we define $V': M\to M$ by
\begin{align}\label{Eq_V'}
 V'(e_{\tau,j})= F(\check{e}_{\tau,j})^{\smvee}, 
\end{align}

\begin{lemma}\label{FV'}
Let $F:M\to M$ as in \autoref{Eq_F} and $V': M\to M$ as in \autoref{Eq_V'}.
Then   
\begin{enumerate}[leftmargin=*]
    \item $\ker(V')=Q$;
    \item $\rm{Im}(F) \cap \rm{Im}(V')=\{0\}$;
    \item if $W\subseteq M_{\tau}$ is  spanned by a subset of $\{e_{\tau,1}, \dots e_{\tau, g(\OO)}\}$, then $(F+V')(W)=F(W)\oplus V'(W)$.
\end{enumerate}
\end{lemma}
\begin{proof}
\hfill
\begin{enumerate}[leftmargin=*]
    \item For $x \in M$, we have $V'(x)=0$ if and only if $F(\check{x})=0$. This is equivalent to $\check{x}\in \check{Q}$, which happens if and only if $x\in Q$. Therefore, $\ker(V')=Q$.
    \item By definition, $\im(V')\subseteq \im(F)^{\smvee}$ and $\im(F)\cap \im(F)^{\smvee}=\{0\}$.
    \item From \autoref{Eq_F} and \autoref{Eq_V'} combined,
\begin{equation*}
    (F+V')(e_{\tau,j})=\begin{cases}
     F(e_{\tau,j})  &\text{if }  \wt(j)>f(\tau)\\
     V'(e_{\tau,j})  &\text{if } \wt(j)\leq f(\tau).
    \end{cases}
\end{equation*}
This implies that $(F+V')(W)=F(W)+V'(W)$. Hence, the statement follows from part 2.
\end{enumerate}
\end{proof}

For $\tau\in\OO$ and $1\leq j_2\leq j_3\leq g(\OO)$, we denote $M_{\tau,j_2,j_3}=\Span\{e_{\tau,j_2}, \dots, e_{\tau,j_3}\}$. 

\begin{lemma}\label{V'stream} 
Let $F:M\to M$ as in \autoref{Eq_F} and $V':M\to M$ as in \autoref{Eq_V'}. Let $\tau\in \OO$. Assume $\wt$ is maximal. Then,
\begin{enumerate}[leftmargin=*]
    \item for $\fts \leq j_1<j_2\leq j_3$: $V'(M_{\tau,j_1})=M_{p\tau,\fts+1,j_1}$, and $V'(M_{\tau,j_2,j_3})=M_{p\tau,j_2,j_3}$. 
\item for  $1\leq j\leq g(\OO)$: $(F+V')(M_{\tau,j})=M_{p\tau,j}$.
\end{enumerate}
\end{lemma}
\begin{proof}
By Lemma \ref{length}, since the word $w_{\tau}$ is maximal, we have $i_{\tau,t}=t$ for $1\leq t\leq \fts$ and $j_{\tau,t}=f(\tau^*)+t$ for $1\leq t\leq f(\tau)$.
Now, for $1\leq t\leq f(\ts)$, we have $F(e_{\tau,t})=e_{p\tau,t}$ and $V'(e_{\tau,t})=0$. Whereas for $\fts<t\leq j$, we have $F(e_{\tau,t})=0$ and $V'(e_{\tau,t})=e_{p\tau,t}$.
The result follows.
\end{proof}

Now we fix an orbit $\OO$. Let $l(\OO) \colonequals |\OO|$ denote the length of the orbit. When $\OO$ is fixed, we write $l$ for $l(\OO)$.
\begin{proposition}\label{criteriasecond}
For $\tau \in \OO, 0\leq i\leq l-1$, we define $H_{\tau,i}: M_{p^i\tau} \to M_{p^{i+1}\tau}$ as
\begin{equation}\label{Eq_H}
    H_{\tau,i}(x)=\begin{cases}
    F(x) &\text{if } f(p^i\tau^*)\geq f(\tau^*)\\
    F(x)+V'(x) &\text{if } f(p^i\tau^*)< f(\tau^*).
    \end{cases}
\end{equation}
Suppose
\begin{enumerate}[leftmargin=*]
    \item \label{assumption 1}$\dim(\pi_{\tau}\circ H_{\tau,l-1}\circ\dots H_{\tau,1}\circ H_{\tau,0}(M_{\tau}))=f(\tau^*)$;
    \item \label{assumption 2} for any $\tau'\in\OO$ satisfying $f(\tau'^*)<f(\tau^*)$, we have $w_{\tau'}$ maximal.
\end{enumerate}
Then  
\begin{enumerate}[leftmargin=*]
   \item \label{part 1} for any $0\leq i\leq l-1$:  $H_{\tau,i}\circ \dots H_{\tau,1}\circ H_{\tau,0}(M_{\tau})=M_{p^{i+1}\tau,f(\tau^*)}$;
    \item \label{part 2}$w_{\tau}$ is maximal;
\item\label{part 3} $w_{\tau'}$ is  maximal for any $\tau'\in\OO$ satisfying $f(\tau'^*)=f(\tau^*)$.

\end{enumerate}
\end{proposition}
\begin{proof}
     We prove \autoref{part 1}  by induction on $i$. First, consider the base case $i=0$. We have $H_{\tau,0}=F$. Since $\dim(F(M_{\tau}))=\fts$, by Lemma \ref{F}, we have $F(M_{\tau})=M_{p\tau,\fts}$. Next, suppose that for some $i\geq 1$, we have $H_{\tau,i-1}\circ \dots H_{\tau,1}\circ H_{\tau,0}(M_{\tau})=M_{p^{i}\tau,f(\tau^*)}$. We distinguish two cases.
     
     Suppose $f(p^{i}\ts)\geq \fts$. Thus $H_{\tau,i}=F$. Let $d=\dim(F(M_{p^{i}\tau,f(\tau^*)}))$. Then $d \leq \dim(M_{p^{i}\tau,f(\tau^*)})=\fts$. By Lemma \ref{F}, $F(M_{p^{i}\tau,f(\tau^*)})=M_{p^{i+1}\tau,d}$.  We deduce that  $d=\fts$ from the inequality 
    \begin{equation*}
        \begin{split}
         \fts&=\dim(\pi_{\tau}\circ H_{\tau,l-1}\circ\dots H_{\tau,1}\circ H_{\tau,0}(M_{\tau}))\\
         &=\dim(\pi_{\tau}\circ H_{\tau,l-1}\circ\dots \circ H_{\tau,i+1}(M_{p^{i+1}\tau,l}))
         \leq \dim(M_{p^{i+1}\tau,d})=d.
        \end{split}
    \end{equation*}
   
   Suppose $f(p^{i}\ts)< \fts$. Then $H_{\tau,i}=F+V'$ and $w_{p^i\tau}$ is maximal. By Lemma \ref{V'stream}, we have
    $$H_{\tau,i}\circ \dots H_{\tau,1}\circ H_{\tau,0}(M_{\tau})=(F+V')(M_{p^{i}\tau,f(\tau^*)})=M_{p^{i+1}\tau,\fts}.$$
This completes the induction step and hence the proof of \autoref{part 1}.

We prove \autoref{part 2}. By \autoref{part 1}, $H_{\tau,l-1}\circ \dots H_{\tau,1}\circ H_{\tau,0}(M_{\tau})=M_{\tau,f(\tau^*)}$. Hence by assumption (\ref{assumption 1}),  $\dim(\pi_{\tau}(M_{\tau,\fts}))=\fts$. We deduce that $M_{\tau,\fts}\cap \ker(\pi_{\tau})=\{0\}$; that is, $M_{\tau,\fts}\cap \ker(F)=\{0\}$. By Lemma \ref{length}, $w_{\tau}$ is  maximal.

We prove \autoref{part 3}. Suppose $\tau'\in \OO$ satisfies $f(\tau')=\fts$. By \autoref{part 2}, it suffices to 
 prove that \[\dim(\pi_{\tau'} \circ H_{\tau',l-1} \circ \dots \circ H_{\tau',0}(M_{\tau'}))=f(\tau^*).\]
Write $\tau'=p^i\tau$, for some $1\leq i\leq l-1$.
Then $H_{\tp,j}=H_{\tau,j+i}$ for  $0 \leq j \leq l-i-1$, and $H_{\tp,j}=H_{\tau, j+i-l}$ for  $l-i \leq j \leq l-1$.
Since $H_{\tp,0}=F$, we deduce $$H_{\tp,0}(M_{\tp})=F(M_{\tp})=M_{p\tp,\fts}=H_{\tau,i}\circ \dots H_{\tau,1}\circ H_{\tau,0}(M_{\tau}).$$
This implies that
$$H_{\tau',l-i-1}\circ \dots H_{\tau',1}\circ H_{\tau',0}(M_{\tau'}) =H_{\tau,l-1}\circ \dots H_{\tau,i+1}\circ H_{\tau,i}\circ \dots H_{\tau,1}\circ H_{\tau,0}(M_{\tau})=M_{\tau,\fts}.$$
Since $\wt$ is maximal, Lemma \ref{F} implies $F(M_{\tau,\fts})=M_{p\tau,\fts}=F(M_{\tau})$, that is, $H_{\tau,0}(M_{\tau,\fts})=H_{\tau,0}(M_{\tau})$. We deduce
\begin{align*}
  H_{\tau',l-1}\circ \dots H_{\tau',1}\circ H_{\tau',0}(M_{\tau'}) &=H_{\tau',l-1}\circ \dots\circ H_{\tau',l-i}(M_{\tau,\fts})\\
&=H_{\tau,i-1}\circ \dots\circ H_{\tau,0}(M_{\tau,\fts})
    = H_{\tau,i-1}\circ \dots\circ H_{\tau,0}(M_{\tau}).
\end{align*}
Since $\ker(\pi_{\tau'})=Q_{\tau'^*}^{\smvee}=\ker(F_{\tau'})$, all subspaces $W\subseteq M_{\tau'}$ satisfy $\dim(\pi_{\tau'}(W))=\dim(F(W))$. Hence,
\begin{align*}
    \dim(\pi_{\tau'} \circ H_{\tau',l-1} \circ \dots \circ H_{\tau',0}(M_{\tau'}))
    &= \dim(F(H_{\tau,i-1}\circ \dots \circ H_{\tau,0}(M_{\tau})))\\
    &=\dim(H_{\tau,i}\circ H_{\tau,i-1}\circ \dots\circ H_{\tau,0}(M_{\tau}))\\
    &=\dim(M_{p^{i+1}\tau,\fts})=\fts. \qedhere
\end{align*}
\end{proof}

Let $f_{1}<f_{2}<\dots< f_{s(\OO)}$ denote the distinct values in $\FF(\OO)=\{\fts\mid \tau\in \OO\}$. For each $1\leq u\leq s(\OO)$, we choose $\tau_{u}\in\OO$ satisfying $f(\tau_{u}^*)=f_{u}$.
From Proposition \ref{criteriasecond}, we  deduce the main result of this section.

\begin{theorem}\label{numerical criteria}
  Fix an orbit $\OO$ and let $l=|\OO|$. For $\tau\in \OO$ and $0\leq i\leq l-1$,  denote $H_{\tau,i}: M_{p^i\tau} \to M_{p^{i+1}\tau}$ as in \autoref{Eq_H}. Let $1\leq u\leq s(\OO)$, and assume that for each $1\leq j\leq u$: we have $\dim(\pi_{\tau_{j}}\circ H_{\tau_{j},l-1} \circ \dots \circ H_{\tau_{j},0}(M_{\tau_{j}}))=f_{j}.$
Then  $w_{\tau}$ is maximal for all $\tau\in \OO$ satisfying $f(\tau^*)\leq f_{u}$. 
\end{theorem}

\section{The Hasse-Witt triple of cyclic covers of $\Po$}\label{duality}

The goal of this section is to explicitly compute the Hasse-Witt triple of a cyclic cover of $\Po$.

In \cite{Moonen algorithm}, Moonen gives an explicit algorithm for computing the Hasse--Witt triple of a complete intersection curve defined over $\fpb$. 
By adapting \cite[Proposition 3.11 and Formula (3.11.3)]{Moonen algorithm} 
to the special case of a cover of $\Po$, we obtain the following description of its Hasse-Witt triple.

\begin{proposition}\label{P_Hasse Witt triple} (Special case of \cite[Proposition 3.11]{Moonen algorithm})
Let $\pi: C \to \Po$ be a smooth projective branched cover of the projective line.  The Hasse-Witt triple  of $C$ is $(Q,\phi,\psi)$ where
\begin{enumerate}[leftmargin=*]
    \item $Q=\hy$, and $Q^{\smvee}=\hyd$;
    \item $\phi:\hy \to \hy$ is given by the Hasse-Witt matrix;
    \item $\psi: \ker(\phi) \to \im(\phi)^\perp$ is defined as $\psi(\alpha)=(df_{1,{\alpha}},-df_{2,{\alpha}}),$
    where $(df_{1,{\alpha}},-df_{2,{\alpha}})$ denotes the global $1$-form on $C$ which restricts to $df_{1,{\alpha}}$ on $U_1$ and to $-df_{2,{\alpha}}$ on $U_2$, for $f_{1,{\alpha}} \in \OO_C(U_1)$ and $f_{2,\alpha} \in \OO_C(U_2)$ satisfying $\alpha^p=f_{1,{\alpha}}+f_{2,{\alpha}}$. 
\end{enumerate}

\end{proposition}

\begin{proof}
   The statement follows by adapting \cite[proof of Proposition 3.11 and Formula (3.11.3)]{Moonen algorithm} to our context, and using the Cech complex with the coordinate charts $U_1=\pi^{-1}(\Po-\{\infty\})$ and $U_2=\pi^{-1}(\Po-\{0\})$.
   Note that if  $\alpha \in \ker(\phi)$, then there exists $f_{1,\alpha}\in \OO_C(U_1)$ and $f_{2,\alpha}\in \OO_C(U_2)$ satisfying $\alpha^p=f_{1,{\alpha}}+f_{2,{\alpha}}$. By construction, $df_{1,{\alpha}}$ and $-df_{2,{\alpha}}$ agree on $U_1\cap U_2$.
\end{proof}

In the following, for $C$ a cyclic cover of $\Po$, we compute the extended Hasse--Witt matrix of $C$, that is, the matrices of the maps $\phi,\psi$ with respect to an explicit choice of bases of $Q$ and $Q^{\smvee}$, for $(Q,\phi,\psi)$ the Hasse--Witt triple of $C$.

\begin{notation}\label{new part mu ord}
Let $G$ be a cyclic group of size $m$. We fix an identification $G=\Z/m \Z$ and we write a monodromy datum for $G$ as $(m,r,\au)$, where $r\geq 3$ and $\au=(a_1,\dots a_r) \in \Z^r$ satisfies
\begin{enumerate}[leftmargin=*]
\item $1 \leq a_k \leq m-1$, for all $1\leq k\leq r$;
\item $\sum_{k=1}^r a_k \equiv 0 \pmod{m}$; 
\item $\gcd(a_1, \dots, a_r,m)=1$. 
\end{enumerate}
Let $\zeta_m=e^{2\pi i/m}\in\C^*$ is a primitive root of $m$. 
We also identify $\tg=\Z/m\Z=\{b\in \Z\mid 0\leq b\leq m-1\}$ by $\tau_b(a)=\zeta_{m}^{ba}.$ Under this identification,  $\tau_b\in \tg^{\n}$ if and only if $\gcd(b,m)=1$.

Let $p$ be a rational prime, $p\nmid m$. 
Given a cyclic monodromy datum $(m, r,\au)$, we consider $C$ the smooth projective curve over $k=\fpb$, which is the normalization of the affine equation $$C':y^m=(x-x_1)^{a(1)} \cdots (x-x_r)^{a(r)},$$ where we assume $x_1, \dots , x_r\in k$ distinct.

In the following, we sometime write the affine equation of $C'$ as $y^m=f(x)$ or $f(x,y)=0$. 
\end{notation}

Let $C/k$ as in Notation \ref{new part mu ord}. 
We denote by $\langle\, ,\, \rangle$  the duality between $\hy$ and $\hyd$, that is
$$\langle f,\omega\rangle=\sum_{P \in C} \Res_P f\omega, \text{ for }  f \in \hy, \omega \in \hyd.$$

For $1 \leq i \leq m-1$, set 
$$v_i=\frac{y^i}{(x-x_1)^{\fy} \dots (x-x_r)^{\fr}}.$$
Note that $v_i$ only depends on the congruence class of $i$ mod $m$.

In the following, we compute $\hy$ via C{e}ch cohomology, with the affine charts $U=\pi^{-1}(\Po-\{\infty\})$ and $V=\pi^{-1}(\Po-\{0\})$. That is, we identify $$\hy = \frac{\OO_C(U \cap V)}{\ocu \oplus \ocv}.$$

\begin{lemma}\label{L_bases}
Let $C/k$ as in Notation \ref{new part mu ord}. Denote by $(Q,\phi,\psi)$ the Hasse--Witt triple of $C$. That is, let $Q=\hy$ and identify $Q^{\smvee}=\hyd$. 
For $1\leq i\leq m-1$, denote by $Q_{\tau_i}$ the $\tau_i$-isotypic component of $Q$. Write $Q^{\smvee}_{\tau_i}=(Q_{\tau_i})^{\smvee}$ (hence, $Q^{\smvee}_{\tau_i}=(Q^{\smvee})_{\tau^*_i}$).

For any $1\leq i\leq m-1$,
\begin{enumerate}[leftmargin=*]
\item the set $\{ \zij=\frac{v_i}{x^j}\mid 1\leq j\leq f(\tau^*_i) \}$ is a $k$-basis basis of $Q_{\tau_i}$ as a $k$-vector space;

\item the set $\{\omega_{i,j}=\frac{x^{j-1}}{v_i}dx \mid 1 \leq j \leq f({\tau_i}^*)\}$ is a basis of $Q^{\smvee}_{\tau_i}$ as a $k$-vector space.
\end{enumerate}
Furthermore, for any $1\leq i,i'\leq m-1, 1\leq j\leq f(\tau^*_i)$ and $ 1 \leq j' \leq f({\tau_{i'}}^*)$, we have
\begin{equation*}
    \langle\zij,\omega_{i',j'}\rangle = 
    \begin{cases}
    l & \text{if $i=i',j=j'$},  \\
    0 & \text{otherwise}.
    \end{cases}
\end{equation*}

\end{lemma}

\begin{proof}
By definition, the rings of regular functions on $U$, $V$ and $U\cap V$ are 
\begin{align*}
\OO_C(U) &\cong \bigoplus_{i=1}^{m-1}k[x]v_i,
&
\OO_C(V) &\cong \bigoplus_{i=1}^{m-1}k\Big[\frac{1}{x}\Big]x^{-f(\tis)-1}v_i,
&
\OO_C(U \cap V)&\cong \bigoplus_{i=1}^{m-1}k\Big[x,\frac{1}{x}\Big]v_i.
\end{align*}
We deduce 
\begin{equation}\label{eq: ho}
    \begin{split}
        \hy &= \frac{\OO_C(U \cap V)}{\ocu \oplus \ocv} 
        = \bigoplus_{i=1}^{m-1} \bigoplus_{j=1}^{f({\tau_i}^*)}k \big\{\zij\big\} \text{ with }\zij=\frac{v_i}{x^j},
    \end{split}
\end{equation}
The basis of $Q^{\smvee}$ with the above property is given in \cite[Section 5]{Dual basis}.
\end{proof}

Let $1\leq i\leq m-1$. Recall  $\phi_{\tau_i}$ denotes the restriction of $\phi$ to $Q_{\tau_i}$, $\phi_{\tau_i}:Q_{\tau_i}\to Q_{\tau_{pi}}$. 
In \cite[Lemma 5.1]{Irene}, Bouw computes the Hasse-Witt matrix of $C$, that is the matrix of $\phi$ with respect to the basis of $Q$ given in Lemma \ref{L_bases}.
The coefficient of $\zpij$ in $\phi_{\tau_i}(\zij)$ is \begin{align}
\label{coefficient of phii}
(-1)^N \sum_{n_1+ \dots + n_r=N}\binom{\pby}{n_1} \dots \binom{\pbr}{n_r} x_1^{n_1} \dots x_r^{n_r},    
\end{align} 
where $N=\sum_{k=1}^r \pbk -(jp-j')=p(f({\tau_i}^*)+1)-(f(p{\tau_i}^*)+1)-jp+j'$.

Recall from Remark \ref{Rem_EOHassetrople}, that given a Hasse--Witt triple, the associated mod-$p$ Diedonn\'e module depends on the choice of an extension of $\psi$ to $Q$.
More precisely, it depends on
the choice of a complement $R_0$ to $R_1=\ker(\phi_{\tau})$ in $Q$, as $\psi$ is extended to  $Q$  by setting $\psi$ to vanish on $R_0$.  (In \cite{Moonen algorithm}, Moonen  proves
that different choices of $R_0$ give rise to isomorphic mod-$p$ Dieudonn\'e modules.)

Let  $\psi_{\tau_i}$ denote the restriction of $\psi$ to $Q_{\tau_i}$, $\psi_{\tau_i}:  Q_{\tau_i} \supseteq \ker(\phi_{\tau_i})\to {\rm Im}(\phi_{\tau_i})^\perp\subseteq Q_{\tau_{pi}}^{\smvee} .$
With abuse of notation, we denote by
$\psii: Q_{\tau_i}\to Q_{\tau_{pi}}^{\smvee}$ any extension of $\psii$ obtained by extension by zero from a choice of complement of $\ker(\phi_{\tau_i})$ in $Q_{\tau_i}$.

\begin{remark}\label{strategy}
In the following,  to compute the matrix of $\psii$, instead of working with a basis of $\ker(\phi_{\tau_i})$ and a choice of a complement, we define and compute the coefficients of an auxiliary linear map $\psi_{\tau_i}': Q_{\tau_i} \to Q_{p\tau_i^*}^{\smvee}$. We show that the restriction of $\psi_{\tau_i}'$ to $\ker(\phi_{\tau_i})$ agrees with $\psi_{\tau_i}$, and  is injective. In Proposition \ref{When is psii' correct}, we show that under suitable conditions, 
$\qi=\ker(\psi_{\tau_i}')\oplus \ker(\phi_{\tau_i})$.  
\end{remark}

\begin{definition}\label{Def_valid}
    For any $\tau\in \tg$, we say that a map $\psi_{\tau}':Q_\tau\to Q_{p\tau^*}^{\smvee}$ is a {valid extension} of $\psi_{\tau}$ if ${\psi_{\tau}'}_{\vert \ker(\phi_{\tau})}=\psi_{\tau}$ and $Q_\tau=\ker(\psi_{\tau}')\oplus \ker(\phi_{\tau})$. 
    \end{definition}
  If $\psi_\tau'$ is a valid extension for all $\tau\in \tg$, then $\psi'$ is the extension by zero of $\psi$ associated with $R_0=\ker(\psi')$.

We define $\psi_{\tau_i}': Q_{\tau_i} \to Q_{p\tau_i^*}^{\smvee}$ as follows. Consider the inclusion of $k$-vector spaces
\begin{align*}
    \iota& : H^1(C,\OO_C) \to \ocuv,\quad  \zij\mapsto \frac{v_i}{x^j},
\end{align*}
which identifies $H^1(C,\OO_C)$ with the subspace of $\ocuv$ given in (\ref{eq: ho}). This yields a decomposition as $k$-vector space
\begin{align*}
   \ocuv &=\ocu\oplus \iota(H^1(C,\OO_C)) \oplus \ocv  .
\end{align*}
We denote by
the corresponding projections of $k$-vector spaces; they satisfy $\ker(\pi_1)=\bigoplus_{i=1}^m k[\frac{1}{x}]x^{-1}v_i$, and $\ker(\pi_2)=\bigoplus_{i=1}^m k[x]x^{-\fis}v_i$.

For any $\alpha \in Q_{\tau_i}$, let $f_1=\pi_1(\iota(\alpha)^p)\in \ocu$, and define
\begin{equation}\label{Eq psi'}
\begin{split}
    \psi_{\tau_i}'(\alpha) & = \sum_{k=1}^{f(p\tau_i)}\langle df_{1}, \zpisk\rangle\opisk 
    =\sum_{k=1}^{f(p\tau_i)}( \text{coefficient of $\opisk$ in $df_{1}$})\opisk.
\end{split}
\end{equation}

\begin{proposition}\label{coefficient of psii}
Let $1\leq i\leq m-1$. Assume $\gcd(i,m)=1$. Then, the coefficient of $\opisjp$ in $\psiip(\zij)$ is 
$$-\sum_{k=1}^r \pbk r_{i,j,k}q_{r-j',k},$$
where $q_{r-j',k}$ is the coefficient of $x^{j'-1}$ in $\frac{(x-x_1) \dots (x-x_r)}{(x-x_k)}$,
and $r_{i,j,k}$ is the residue of $\frac{\zij^p}{v_{pi}(x-x_k)}$ at $x=0$, that is 
    $$r_{i,j,k}=(-1)^N\sum_{n_1+ \dots +n_r=N}\binom{\pby}{n_1} \dots \binom{\pbk -1}{n_k} \dots \binom{\pbr}{n_r} x_1^{n_1}\dots x_r^{n_r},$$
where $N=p(f(\tau_i^*)+1)-(f(p\tau_i^*)+1)-pj$.

\end{proposition}
\begin{proof}
Let $\iota(\zij)^p=f_{1,j}+\iota(\phi(\zij))+f_{2,j}$, where $f_{1,j} \in \ocu$ and $f_{2,j} \in \ocv$. We have
$$\iota(\zij)^p=\vip \frac{(x-x_1)^{\pby} \dots (x-x_r)^{\pbr}}{x^{pj}}.$$
Then $f_{1,j}=h(x)\vip$ where
$$h(x)=\left[\frac{\pbf}{x^{pj}}\right]_{\text{deg $\geq 0$}}.$$
We compute the coefficient of $\opisjp$ in $d(h(x)\vip)$.
Since we are in characteristic $p$, 
\begin{equation*}
    \begin{split}
        d(\vip)&=y^{pi}d\left(\frac{1}{(x-x_1)^{\pfy}\dots (x-x_r)^{\pfr}}\right) \\
        &=\frac{y^{pi}}{(x-x_1)^{\pfy}\dots (x-x_r)^{\pfr}}\bigg(\frac{-\pfy}{(x-x_1)}+ \dots +\frac{-\pfr}{(x-x_r)}\bigg)dx \\
        &=v_{pi}\bigg(\frac{-\pfy}{(x-x_1)}+ \dots +\frac{-\pfr}{(x-x_r)}\bigg)dx 
        = v_{pi}\bigg(\frac{-\pby}{(x-x_1)}+ \dots +\frac{-\pbr}{(x-x_r)}\bigg)dx,
    \end{split}
\end{equation*}
where the last equality follows from the equality
$\pfk-\pbk = \floor*{p\fk}=p\fk \equiv 0 \pmod{p}.$
Assume
\begin{equation}\label{eqn: *}
d(h(x))=\sum_{k=1}^r \pbk \frac{h(x)-r_{i,j,k}}{x-x_k} dx. \tag{*}  
\end{equation}
Then
\begin{equation*}
    \begin{split}
    d(f_{1,j}) &= d(\vip)h(x)+\vip d(h(x)) 
    =- \vip \sum_{k=1}^r \pbk \frac{h(x)}{x-x_k} dx + \vip \sum_{k=1}^r \pbk \frac{h(x)-r_{i,j,k}}{x-x_k} dx\\
    &=-\vip \sum_{k=1}^r \pbk \frac{r_{i,j,k}}{x-x_k} dx.
    \end{split}
\end{equation*}
Since $\gcd(i,m)=1$, we have $\vip\vips=(x-x_1) \dots (x-x_r)$ and hence $\vip dx = \frac{(x-x_1) \dots (x-x_r)}{\vips} dx$. Thus,
$$d(f_{1,j})=- \sum_{k=1}^r \pbk r_{i,j,k}\frac{(x-x_1) \dots (x-x_r)}{(x-x_k)} \frac{dx}{\vips}.$$
Therefore, the coefficient of $\opisjp$ is equal to that of $x^{j'-1}$ in $-\sum_{k=1}^r \pbk r_{i,j,k}\frac{(x-x_1) \dots (x-x_r)}{(x-x_k)}$. 

We reduced the statement to the proof of the equality in (*).

Let  $h_1(x)=\frac{\iota(\zij)^p}{\vip}=\frac{\pbf}{x^{pj}}.$ Then,  $d(h_1(x))=\sum_{k=1}^r \pbk \frac{h_1(x)}{x-x_k} dx$, and the residue at $0$ of $\frac{h_1(x)}{x-x_k}$ is equal to  $r_{i,j,k}$.

We write 
$\frac{h_1(x)}{x-x_k}=h_{1,k}(x)+\frac{r_{i,j,k}}{x}+g_{1,k}(x),$
where $h_{1,k}(x)=\left[\frac{h_1(x)}{x-x_k}\right]_{\text{deg $ \geq 0$}}$ and $g_{1,k}(x)$ consists of all terms of degree $\leq -2$.
Then,  $h_1(x)=(x-x_k)h_{1,k}(x)+r_{i,j,k}-\frac{x_kr_{i,j,k}}{x}+(x-x_k)g_{1,k}(x)$.

By definition, $h(x)=\left[ h_1(x)\right]_{\text{deg $ \geq 0$}}$. Hence, $h(x)=(x-x_k)h_{1,k}(x)+ r_{i,j,k}$, and
$$d(h(x))=[d(h_1(x))]_{\text{deg $\geq 0$}}=\sum_{k=1}^r \pbk \left[\frac{h_1(x)}{x-x_k}\right]_{\text{deg $ \geq 0$}} dx.$$
\end{proof}

\begin{lemma}
    For any $1\leq i\leq m-1$, if $\alpha \in \ker(\phi_{\tau_i})$, then $\psi_{\tau_i}'(\alpha)=\psi_{\tau_i}(\alpha)$. 
\end{lemma}
\begin{proof}
By definition,   $\iota(\alpha)^p=f_1+\iota(\phi(\alpha))+f_2$, where $f_i=\pi_i(\iota(\alpha)^p)$, for $1\leq i\leq 2$. Hence, the statement follows by comparing the definition of $\psii'$ in \autoref{Eq psi'} with that of $\psii$ in Proposition \ref{P_Hasse Witt triple}.
\end{proof}

\begin{lemma}\label{psii' injective}
For any $1\leq i\leq m-1$, $\ker(\phi_{\tau_i})\cap\ker(\psi'_{\tau_i})=\{0\}.$
\end{lemma}

\begin{proof}

Let $\alpha \in \ker(\phi_{\tau_i}) \cap \ker(\psi'_{\tau_i})$, and write
$\iota(\alpha)=c_1\frac{v_i}{x}+c_2\frac{v_i}{x^2}+ \dots +c_{f(\tau_i^*)}\frac{v_i}{x^{f(\tau_i^*)}}$.
Then,
  $$    \iota(\alpha)^p
      =v_{pi}(x-x_1)^{\pby} \dots (x-x_r)^{\pbr}\left(\frac{c_1x^{f(\tau_i^*)-1}+c_2x^{f(\tau_i^*)-2}+ \dots +c_{f(\tau_i^*)}}{x^{f(\tau_i^*)}} \right)^p.$$ 
   Set $ h(x)=c_1x^{f(\tau_i^*)-1}+c_2x^{f(\tau_i^*)-2}+ \dots +c_{f(\tau_i^*)} $ and $$ \mathcal{H}(x)=\frac{(x-x_1)^{\pby}\dots (x-x_r)^{\pbr}h(x)^p}{x^{pf(\tau_i^*)}}.$$
Write $g(x)= \left[\mathcal{H}(x)\right]_{\deg \geq 0}$ and 
    $q(\frac{1}{x})=\left[\mathcal{H}(x)\right]_{\deg \leq -f(p\tau_i^*)-1}.$

Since $\phi(\alpha)=0$,  $\iota(\alpha)^p=g(x)v_{pi}+q(\frac{1}{x})v_{pi}$, and hence $0=d(g(x)v_{pi})+d(q(\frac{1}{x})v_{pi})$.
Since $\psi(\alpha)=0$, we have $d(g(x)v_{pi})=0$ and $d(q(\frac{1}{x})v_{pi})=0$. 
Hence, to deduce that $\alpha=0$, it suffices to prove that $h(x)=0$. 

Assume by contradiction $h(x) \neq 0$. 
Let 
$0\leq b\leq c\leq f(\tau_i^*)-1$, denote respectively the smallest and largest integer $n$ for which coefficient of $x^{n}$ in $h(x)$ is non-zero.

Then, the highest power of $x$ in $\mathcal{H}(x)$ is $x^{p(f(\tau_i^*)+1)-f(p\tau_i^*)-1+pc-pf(\tau_i^*) }=x^{p(c+1)-1-f(p\tau_i^*)}$. Since the exponent is strictly positive, $g(x)\neq 0$. Similarly, the lowest power of $x$ in $\mathcal{H}(x)$ is $x^{bp-pf(\tau_i^*)}$, and since the exponent is negative, $q(\frac{1}{x})\neq 0$.

Let $1\leq k\leq m-1$ such that $pi \equiv k \pmod{m}$. Then, by definition, $v_{pi}=v_k=\frac{y^k}{(x-x_1)^{\fky} \dots (x-x_r)^{\fkr}}$.
Set 
$$g_1(x)=\frac{g(x)}{(x-x_1)^{\fky}\dots (x-x_r)^{\fkr}}=\frac{g(x)v_k}{y^k}.$$
Then, $g(x)v_{k}=y^kg_1(x)$, and $d(y^kg_1(x))=d(g(x)v_{pi})=0$.

Denote by $y^m=f(x)$ the affine equation of the curve. Then, $my^{m-1}dy=f'(x)dx$, and
\begin{equation*}
    \begin{split}
        0&=d(y^kg_1(x)) =ky^{k-1}g_1(x)dy+y^kg_1'(x)dx =\frac{ky^{k-1}g_1(x)f'(x)}{my^{m-1}}dx +y^{k}g_1'(x) dx \\
        &=\frac{dx}{y^{m-k}}\left(\frac{k}{m}f'(x)g_1(x)+y^mg_1'(x)\right) =\frac{dx}{y^{m-k}}\left(\frac{k}{m}f'(x)g_1(x)+f(x)g_1'(x)\right),
    \end{split}
\end{equation*}
which implies that  $\frac{k}{m}f'(x)g_1(x)+f(x)g_1'(x)=0$. We deduce
$$\left(f(x)^kg_1(x)^m\right)'=\left(mf(x)^{k-1}g_1(x)^{m-1}\right)\left(\frac{k}{m}f'(x)g_1(x)+f(x)g_1'(x)\right)=
0.$$

Hence, $f(x)^kg_1(x)^m=s(x)^p$ for some  rational function $s(x)$.
We deduce \begin{equation}\label{eq1}
   (x-x_1)^{m\langle\frac{ka_1}{m}\rangle} \dots (x-x_r)^{m\langle\frac{ka_r}{m}\rangle}g(x)^m=\prod(x-\beta_s)^{\pm p},
\end{equation}
where $\beta_s$ are the zeros and poles of $s(x)$. 
Similarly, we deduce an analogous expression for $q(\frac{1}{x})$,
\begin{equation}\label{eq2}
(x-x_1)^{m\langle\frac{ka_1}{m}\rangle} \dots (x-x_r)^{m\langle\frac{ka_r}{m}\rangle}q(\frac{1}{x})^m=\prod(x-\gamma_s)^{\pm p}.
\end{equation}
Let $a(x)=x^{pf(\tau_i^*)}q(\frac{1}{x})$. Then, $a(x)$ is a polynomial and $\deg(a(x)) \leq pf(\tau_i^*)-f(p\tau_i^*)-1$.

From \eqref{eq1} and
\eqref{eq2}, we deduce $v_{(x-\alpha)}(g(x)) \equiv 0 \pmod{p}$ and $v_{(x-\alpha)}(a(x)) \equiv 0 \pmod{p}$, for any $\alpha \not \in \{x_1, \dots, x_r\}$, and 
for any $1\leq j\leq r$, and for $t_j$ and $ m_j$  the powers of $x-x_j$ in $g(x)$ and  $a(x)$ respectively, 
$m\langle\frac{ka_j}{m}\rangle+t_jm \equiv 0 \pmod{p}\text{ and }
m\langle\frac{ka_j}{m}\rangle+m_jm \equiv 0 \pmod{p}
$. 
We deduce $t_j \equiv m_j \pmod{p}$.

Recall 
$(x-x_1)^{\pby} \dots (x-x_r)^{\pbr} h(x)^p=x^{pf(\tau_i^*)}g(x)+a(x).$
We deduce $ \pbj \equiv t_j\equiv m_j  \pmod{p}$, and $t_j,m_j \geq \pbj$. The latter inequality implies the contradiction 
$$p(f(\tau_i^*)+1)-f(p\tau_i^*)-1=\sum_{j=1}^r \pbj \leq \sum_{j=1}^r m_j \leq \deg(a(x)) \leq pf(\tau_i^*)-f(p\tau_i^*)-1.$$
Hence, $h(x) =0$, which concludes the proof.
\end{proof}

\begin{proposition}\label{When is psii' correct}
Let $1\leq i\leq m-1$.  Assume either  $f(p\tau_i^*)=0$ or
$f(\tau_i^*)=g(\tau_i)$.
Then, the map $\psi_{\tau_i}'$ is a valid extension of $\psi_{\tau_i}$, that is  
$Q_{\tau_i}=\ker( \psi_{\tau_i})\oplus \ker(\phi_{\tau_i}').$\end{proposition}
\begin{proof}

Assume $f(p\tau_i^*)=0$. Then $Q_{p\tau_i}=\{0\}$ and since $\phi_{\tau_i}:Q_{\tau_i}\to Q_{p\tau_i}$, we deduce that $\ker(\phi_{\tau_i})=Q_{\tau_i}$, and hence $\ker(\psii')=0$, by Lemma \ref{psii' injective}.

Assume $f(\tau_i^*)=g(\tau_i)$. Recall $g(\tau_i)=g(\tau_{pi})$, since $\tau_i,\tau_{pi}$ are in the same orbit under Frobenius.  
By Lemma \ref{psii' injective},
it suffices to 
show that $\dim {\rm Im}(\psi_{\tau_i}') \leq \dim \ker(\phi_{\tau_i})$.
The inclusion $\psi_{\tau_i}'(Q_{\tau_i}) \subseteq Q_{p\tau_i^*}^{\smvee}$ 
implies $\dim {\rm Im} (\psi_{\tau_i}')\leq \dim(Q_{p\tau_i^*}^{\smvee})= f(p\tau_i)$. On the other hand, 
 $$\dim(\ker(\phi_{\tau_i}))=\dim(Q_{\tau_i})-\dim({\rm Im}(\phi_{\tau_i}))\geq f(\tau_i^*)-\dim(Q_{p\tau_i})= g(\tau_i)-  f(p\tau_i^*)=f(p\tau_i).$$
 Hence, $\dim {\rm Im}(\psi_{\tau_i}')\leq f(p\tau_i)\leq \dim \ker(\phi_{\tau_i})$.
\end{proof}

\section{Entries of the extended Hasse--Witt matrix}\label{max monomial}

In Proposition \ref{coefficient of psii}, we computed the entries of the extended Hasse--Witt matrix as polynomials in  $x_1, \dots, x_r$ and showed that they are homogenous polynomials. In this section, we show that some of these entries don't vanish identically by identifying certain non-zero monomials.

Let $(m,r,a_1,\dots,a_r)$ be a cyclic monodromy datum as in Notation \ref{new part mu ord}.
Let $p$ be a rational prime.
For the remaining of the paper, we assume $p>m(r-2)$. This condition agrees with the assumption in \cite[Theorem 6.1]{Irene}.

\begin{notation}\label{define c C N}
For $1 \leq i \leq m-1$ and $N \in \mathbb{N}$, we define 
 \begin{align*}
     c_i(N)&=\min\{c : \pby + \dots +\pbc >N\}, \\
 C_i(N)&=N-\pby- \dots -\pbcm, \text{where $c=c_i(N)$},\\
  X_i(N)&=x_1^{\pby} \dots x_{c-1}^{\lfloor{p\langle\frac{ia_{c-1}}{m}\rangle\rfloor}}x_c^{C}, \text{where }c=c_i(N), C=C_i(N),\\
 s_i&=\sum_{k=1}^r \pbk. 
 \end{align*}
 Note that $1\leq c_i(N)\leq r$, $0 \leq C_i(N)< \pbc$, and   
$
\Ci\equiv \sum_{k=c}^r\pbk\pmod{p}.
$
\end{notation}

\begin{lemma}\label{cineqr}
For any $j \geq 1$, we have $\ci \neq r$.
\end{lemma}
\begin{proof}
Assume for the sake of contradiction that $\ci=r$. Then $\sum_{k=1}^{r-1}\pbk \leq s_i-pj \leq s_i-p$. So $p \leq \pbr$, which is impossible. 
\end{proof}

We consider the lexicographical order on the monomials such that $x_1>x_2> \dots > x_r$. Given a monomial $X$, we denote by $v_{x_s}(X)$ the power of $x_s$ in $X$, $1\leq s\leq r$. 
In the following, given a polynomial $h(x_1,\dots,x_r)$, we denote by $m(h(x_1,\dots,x_r))$ its maximal monomial without coefficient, and  by $M(h(x_1,\dots,x_r))$ the maximal monomial including the coefficient.

For $1\leq i\leq m-1$, recall $\phii$ and $\psii$ denote the restrictions of $\phi$ and $\psi$ to $Q_{\tau_i}$, respectively. We identify $\phii$ and $\psii$ with their matrices with respect to the bases in Lemma \ref{L_bases}, and denote respectively by $\phii(j',j)$ and $\psiip(j',j)$ the  $(j',j)$-th entries of $\phii$ and $\psiip$,  $1\leq j'\leq f(\tau^*_{pi})$ and $1\leq j\leq f(\tau^*_i)$. 

In \cite[Lemma 6.5]{Irene}, assuming $p>m(r-2)$, Bouw proved \begin{align}\label{maxmonomialbouw}
m(\phii(j',j))=X_i(s_i-pj+j').\end{align} In the following, we identify certain monomials in $\psii'(1,1)$ that have a non-zero coefficient. 

\begin{proposition}\label{maximal monomial}
Assume $p>m(r-2)$.
Consider $1\leq i\leq m-1$ with $\gcd(i,m)=1$. Write $c=c_i(s_i-p)$. Then, the monomial $ x_1x_2 \dots x_{r-1}\XI$ has a non-zero coefficient in $\psii'(1,1)$.
In particular, $\psii'(1,1)$ is a homogeneous polynomial of degree $s_i-p+r-1$ which is not identically zero. 
\end{proposition}
\begin{proof}
For $1\leq k\leq r$, let $q_{r-1,k}$ and $r_{i,1,k}$ be as in Proposition \ref{coefficient of psii}.

%
Then, the maximal monomial in $q_{r-1,k}$ is
\begin{align}\label{mq}
M(q_{r-1,k})=\begin{cases}
(-1)^{r-1}x_1 \dots x_{r-1} &\text{ if } k = r, \\
   (-1)^{r-1}x_1 \dots x_{r-1} \frac{x_{r}}{x_k} &\text{ if }k \leq r-1.
\end{cases}
\end{align}

The maximal monomial in $r_{i,1,k}$  is 
\begin{align}\label{mr}
M(r_{i,1,k})=\begin{cases}
  (-1)^{s_i-p}\binom{\pbc}{C_i(s_i-p)}X_i(s_i-p) &\text{ if }    k > c_i(s_i-p),\\
    (-1)^{s_i-p}\binom{\pbc -1}{C_i(s_i-p)}X_i(s_i-p) &\text{ if }    k = c_i(s_i-p) ,\\
  (-1)^{s_i-p}\binom{\pbc}{C_i(s_i-p)+1}X_i(s_i-p)\frac{x_c}{x_k} &\text{ if }    k < c_i(s_i-p).
\end{cases}
\end{align}
By Lemma \ref{cineqr}, $c_i(s_i-p) \leq r-1$. Then, both $q_{r-1,k}$ and $r_{i,1,k}$ have larger monomials when $k=r$.
We deduce from the equality $\psii'(j',j)=-\sum_{k=1}^r \pbk r_{i,j,k}q_{r-j',k}$ (Proposition \ref{coefficient of psii}) that the coefficient of the monomial $x_1x_2 \dots x_{r-1}X_i(s_i-p)$ in $\psii'(1,1)$ is
$$-(-1)^{r-1+s_i-p}\binom{\pbc}{C_i(s_i-p)}\pbr,$$
which does not vanish modulo $p$.
\end{proof}

\begin{remark}
    In general for $p>m(r-2)$, $1\leq i\leq m-1$ with $\gcd(i,m)=1$, $1\leq j'\leq f(\tau^*_{pi})$ and $1\leq j\leq f(\tau^*_i)$. Write $c=\ci$. Then, the maximal monomial in $\psii'(j',j)$ is
\begin{equation*}
    m(\psii'(j',j)) = 
       \begin{cases}
      x_1x_2 \dots x_{r-j'}\XI & \text{if $\ci \leq r-j'$}\\
     x_1x_2 \dots x_{r-j'-1}x_c\XI & \text{if $\ci > r-j'$}.\\
    \end{cases}
\end{equation*}
However, we will not be needing this more general result, so we omit the proof.
\end{remark}

\section{$p$-ordinariness in special instances}\label{Section 0,1 in signature}

Let $(m,r,a_1,\dots,a_r)$ be a cyclic monodromy datum and $p$ a rational prime, as in Notation \ref{new part mu ord}. 
Assume $p>m(r-2)$.


For $\xb \in \mlra(\fpb)$, let $\ct$ be the normalization of $y^m=(x-x_1)^{a_1} \dots (x-x_r)^{a_r}$, and denote by
$\wt(\ct)$ the Weyl coset representative at the character $\tau$ of the mod-$p$ Dieudonné module of $C_{\tb}$.

\begin{definition}\label{generically maximal}
We say that $\wt$ is generically maximal if there exists a non-empty open subset $U$ of $\mlra(\fpb)$ such that for every $\tb \in U$, $\wt(\ct)$ is maximal.  
\end{definition}

The goal of this section is to prove the following result.

\begin{theorem}\label{signature 1}
Let $\OO$  the Frobenius orbit in $\tg^\n$. 
Assume $\FF(\OO)$ contains $\{0,1\}$. 

If $\tau\in \OO$ satisfies either $f(\ts)=0$ or $f(\ts)=1$,  then  $\wt$ is generically maximal.
\end{theorem}
Note that the above Theorem immediately implies the following corollary.
\begin{corollary}
Let $\OO$ be an Frobenius orbit in $\tg^\n$. 
Assume $\FF(\OO)$ contains $\{r-2,r-3\}$.

If $\tau \in \OO$ satisfies either $\fts =r-2$ or $\fts=r-3$, then $\wt$ is generically maximal.
\end{corollary}
\begin{proof}
If $\FF(\OO)$ contains $\{r-2,r-3\}$, then $\FF(\OO^*)$ contains $\{0,1\}$, and for any $\tau \in \OO$ such that $\fts =r-2$ or $r-3$, $\ft=0$ or $1$. Therefore, $\wts$ is generically maximal. Therefore, $\wt$ is generically maximal. 
\end{proof}

\begin{proof}[Proof of Theorem \ref{signature 1}]
Since $\FF(\OO)\supseteq \{0,1\}$, we have $f_{\OO,1}=0$ and $f_{\OO,2}=1$. Consider $\tau_{1},\tau_{2}\in\OO$ satisfying $f(\tau_{1}^*)=0$ and $f(\tau_{2}^*)=1$. Let $l=l(\OO)$ be the size of the orbit. By Theorem \ref{numerical criteria}, it suffices to show that the following two conditions hold generically: 
\begin{enumerate}[leftmargin=*]
    \item $\dim(\pi_{\tau_{1}}\circ H_{\tau_{1},l-1} \circ \dots \circ H_{\tau_{1},0}(M_{\tau_{1}}))=0$;
    \item
   $\dim(\pi_{\tau_{2}}\circ H_{\tau_{2},l-1} \circ \dots \circ H_{\tau_{2},0}(M_{\tau_{2}}))=1.$
\end{enumerate}
Note that the first condition holds trivially since $\pi_{\tau_{1}}:M_{\tau_{1}}\to Q_{\tau_{1}}$ and $\dim(Q_{\tau_{1}})=f(\tau_{1}^*)=0$. We focus on the second condition. For simplicity, write $\tau=\tau_{2}$, hence $f(\ts)=1$. 

We rewrite the second condition explicitly in terms of $\phi,\psi$.
Recall the notation for Hasse--Witt triple from Section \ref{duality}. For $0 \leq i \leq l-1$, 
consider the map $\check{\phi}_{p^i\tau^*}: Q_{p^i\tau^*}^{\smvee} \to Q_{p^{i+1}\tau^*}^{\smvee}$, defined as $\check{\phi}_{p^i\tau^*}(x)=\phi_{p^i\tau^*}(\check{x})^{\smvee}$, and the map $\check{\psi}_{p^i\tau^*}: Q_{p^i\tau^*}^{\smvee} \to Q_{p^{i+1}\tau^*}$, defined as $\check{\psi}_{p^i\tau^*}(x)=\psi_{p^i\tau^*}(\check{x})^{\smvee}$.
Then, with respect to the choice of bases for $Q_{\tau'}$ and $Q^{\smvee}_{\tau'}$ given in Section \ref{duality}, the matrices representing $\check{\phi}_{p^i\tau^*}$ and $\phi_{p^i\tau^*}$ (respectively,  $\check{\psi}_{p^i\tau^*}, \psi_{p^i\tau^*}$) are the same.
For $0\leq i\leq l-1$, we define
\begin{equation}\label{Eq_A_i}
    A_{i}=
    \begin{cases}
    \phi_{p^i\tau}: Q_{p^i\tau}\to Q_{p^{i+1}\tau} &\text{if } f(p^i\tau^*)\geq 1, f(p^{i+1}\tau^*)\geq 1,\\
    \check{\phi}_{p^i\tau^*}: Q_{p^i\tau^*}^{\smvee} \to Q_{p^{i+1}\tau^*}^{\smvee}  &\text{if } f(p^i\tau^*)=0, f(p^{i+1}\tau^*)=0,\\
    \psi_{p^i\tau}: Q_{p^i\tau}\to Q_{p^{i+1}\tau^*}^{\smvee} &\text{if } f(p^i\tau^*)\geq 1, f(p^{i+1}\tau^*)=0,\\
    \check{\psi}_{p^i\tau^*}: Q_{p^i\tau^*}^{\smvee} \to Q_{p^{i+1}\tau^*} &\text{if } f(p^i\tau^*)=0, f(p^{i+1}\tau^*)\geq 1.
    \end{cases}
\end{equation}
By Proposition \ref{When is psii' correct}, for any $0\leq i\leq l-1$, in the definition of $A_i$, we may replace $\psi$ with $\psi'$.

Note that the composition
$A_{l-1}\circ \dots \circ A_{0} :Q_{\tau}\to Q_{\tau}$
is well defined. 

We claim that the maps
$A_{l-1}\circ \dots \circ A_{0} :Q_{\tau}\to Q_{\tau}$ and 
$\pi_{\tau}\circ L_{\tau,l-1}\circ \dots \circ L_{\tau,0} :Q_{\tau}\to Q_{\tau}$ agree.
By \autoref{Eq_H}, since $\fts=1$, for $0 \leq i \leq l-1$, we have 
\begin{equation*}
    H_{\tau,i} = \begin{cases}
           F_{p^i\tau} &\text{if } f(p^i\tau^*)\geq 1\\
           V'_{p^i\tau} &\text{if } f(p^i\tau^*)=0.
    \end{cases}
\end{equation*}
Here we are using the fact that if $f(\pits)=0$, then $F_{\pits}=0$.

Recall that $F_{p^i\tau}=\phi_{p^i\tau}+\psi_{p^i\tau}$ and $V'_{p^i\tau}=\check{\phi}_{p^i\tau^*}+\check{\psi}_{p^i\tau^*}$.
Hence, for $0\leq i\leq l-2$:
\begin{itemize}[leftmargin=*]
    \item if $f(p^i\tau^*)\geq 1$ and $f(p^{i+1}\tau^*)\geq 1$, then $H_{\tau,i}=F_{p^i\tau}$ and $H_{\tau,i+1}=F_{p^{i+1}\tau}$; moreover,  we have $F_{p^{i+1}\tau}\circ F_{p^{i}\tau}= F_{p^{i+1}\tau}\circ \phi_{p^i\tau}$;
    \item if $f(p^i\tau^*)=0$ and $f(p^{i+1}\tau^*)=0$, then $H_{\tau,i}=V'_{p^i\tau}$ and $H_{\tau,i+1}=V'_{p^{i+1}\tau}$; moreover, we have $V'_{p^{i+1}\tau}\circ V'_{p^{i}\tau}=  V'_{p^{i+1}\tau}\circ \check{\phi}_{p^i\tau^*}$;
    \item if $f(p^i\tau^*)\geq 1 $ and $ f(p^{i+1}\tau^*)=0$, then $H_{\tau,i}=F_{p^i\tau}$ and $H_{\tau,i+1}=V'_{p^{i+1}\tau}$; moreover,  we have $V'_{p^{i+1}\tau}\circ F_{p^{i}\tau}= V'_{p^{i+1}\tau} \circ \psi_{p^i\tau}$;
    \item if $f(p^i\tau^*)=0 $ and $ f(p^{i+1}\tau^*)\geq 1$, then $H_{\tau,i}=V'_{p^i\tau}$ and $H_{\tau,i+1}=F_{p^{i+1}\tau}$; moreover,  we have $F_{p^{i+1}\tau}\circ V'_{p^{i}\tau}= F_{p^{i+1}\tau}\circ \check{\psi}_{p^i\tau^*}$.
\end{itemize}
Finally, for $i=l-1$:  if $f(p^{l-1}\ts)\geq 1$, then $H_{\tau,l-1}=F_{p^{l-1}\tau}$ and $\pi_{\tau}\circ H_{\tau,l-1}=\phi_{p^{l-1}\tau}$;  if $f(p^{l-1}\ts)=0$, then $H_{\tau,l-1}=V'_{p^{l-1}\tau}$ and $\pi_{\tau}\circ H_{\tau,l-1}=\check{\psi}_{p^{l-1}\ts}$. Combined, the aforementioned equalities for $0\leq i\leq l-1$ imply the claim.

Recall $\dim Q_\tau= f(\tau^*)=1$. Hence, we have reduced the statement to show that $A_{l-1}\circ \dots \circ A_{0}:Q_{\tau}\to Q_{\tau}$, when regarded as a polynomial in $\fpb[x_1, \dots, x_r]$, does not vanish.

Let $0\leq i<l$. With respect to the bases given in Lemma \ref{L_bases}, we regard  $A_i$ as a $a(i+1)\times a(i)$ matrix, where $a(i) = f(\pits) $ if $f(p^{i}\tau^*)\geq 1$, and $a(i)= f(\pit) $ otherwise.
We denote by $A_i(j',j)$ the $(j',j)^{th}$ entry of the matrix $A_i$, and by $m_i(j',j)=m(A_i(j',j))$ the maximal monomial in $A_i(j',j)$, for $1\leq j'\leq a(i+1)$ and $1\leq j\leq a(i)$ .

Consider the set $\mathfrak{J}$ consisting of all functions $J:\{0,1, \dots, l\} \to \mathbb{N}$
such that $J(0)=J(l)=1$, and $1 \leq J(i) \leq a(i)$,  for $0<i<l$.
For any $J \in \mathfrak{J}$, set $R_{J,i}=A_i(J(i+1),J(i))$ and $T_{J,i}=m_i(J(i+1),J(i))$; we write 
\begin{equation}
R_{J}=\prod_{i=0}^{l-1}R_{J,i}^{p^{l-i-1}} \text{ and } T_J=\prod_{i=0}^{l-1}T_{J,i}^{p^{l-i-1}}.
\end{equation}
Recall that the maps $A_i$ are $\sigma$-linear, $0\leq i\leq l-1$. Hence, we have
\begin{equation}\label{sumRJ}
   \aip=\sum_{J \in \mathfrak{J}}R_J\in \fpb[x_1,\dots,x_r].
\end{equation}
Also, by definition, $T_{J}$ is the maximal monomial of $R_J$.
In section \ref{max monomial}, we show that for the $J$ such that $J(i)=1$ for all $0 \leq i \leq l$, we have $R_{J,i} \neq 0$. Therefore, for this $J$, $R_J$ and $T_J$ are not identically zero. Hence, by \autoref{sumRJ}, to prove that $\aip$ does not identically vanishes, it suffices to show that, for any $J_1, J_2 \in \mathfrak{J}$ such that $R_{J_1} \neq 0, R_{J_2} \neq 0$, if $T_{J_1}=T_{J_2}$ then $J_1=J_2$.

First, we show that if $T_{J_1}=T_{J_2}$, then $v_{x_s}(T_{J_1,i})=v_{x_s}(T_{J_2,i})$, for all $1 \leq s \leq r$ and $0 \leq i \leq l-1$.
For $1 \leq s \leq r$ and $0 \leq i \leq l-1$, denote
$\eis=v_{x_s}(T_{J_1,i})-v_{x_s}(T_{J_2,i})$. 
Since $T_{J_1}=T_{J_2}$, for each $1 \leq s \leq r$, we have 
$$\sum_{i=0}^{l-1}p^{l-i-1}\eis=v_{x_s}(T_{J_1})-v_{x_s}(T_{J_2})=0.$$
Hence, to prove that $\eis=0$, for all $i,s$, it is enough to prove that  $|\eis|<p$ for all $i,s$. 
We will verify the latter inequalities by direct computations. 
By assumption, $\tau=\tau_b$ for some $1\leq b\leq m-1$ satisfying $\gcd(b,m)=1$. Then, $p^i\tau=\tau_{p^ib}$. We denote
\begin{align*}
\tau_i&=\begin{cases}
    p^i\tau  & \text{if } f(p^i\ts)\geq 1\\
    p^i\ts &\text{if } f(p^i \ts)=0
\end{cases}
,&
b_i&= \begin{cases}
    p^i b  & \text{if } f(p^i\ts)\geq 1\\
    -p^i b &\text{if } f(p^i \ts)=0.
\end{cases}
\end{align*}
Therefore, $A_i$ is one of $\phi_{\tau_i}, \check{\phi}_{\tau_i}$ or $\psi_{\ti}, \check{\psi}_{\ti}$. We distinguish two cases.\\
{\bf Case 1}: Assume $A_i=\phi_{\ti}$ or $\check{\phi}_{\ti}$. Let $M$ be a monomial appearing in $A_i(j',j)$. By \autoref{maxmonomialbouw}, we see that for any $s$, $v_{x_s}(M) \leq \pbis \leq p-2$. In particular, $v_{x_s}(T_{J_1,i}) \leq p-2$, $v_{x_s}(T_{J_2,i}) \leq p-2$. Therefore, $|\eis| \leq p-2< p$. \\
{\bf Case 2}: Assume $A_i=\psi_{\ti}$ or $\check{\psi}_{\ti}$. Let $M$ be a monomial appearing in $A_i(j',j)$. From Proposition \ref{coefficient of psii}, we see that for any $s$, $v_{x_s}(M) \leq \pbis+1 \leq p-1$. In particular, $v_{x_s}(T_{J_1,i}) \leq p-1$, $v_{x_s}(T_{J_2,i}) \leq p-1$. Therefore, $|\eis| \leq p-1< p$. 

Therefore, in all cases, $|\eis| < p$. Hence, if $T_{J_1}=T_{J_2}$, then $\eis=0$ for $0 \leq i \leq l-1$ and $1 \leq s \leq r$. Since $\eis=v_{x_s}(T_{J_1,i})-v_{x_s}(T_{J_2,i})$, we see that $\deg(T_{J_1,i})=\deg(T_{J_2,i})$. We will deduce by backward induction on $i$ that $J_1(i)=J_2(i)$ for all $0 \leq i \leq l $. By definition, $J_1(l)=J_2(l)=1$. Suppose $J_1(i+1)=J_2(i+1)$, we shall deduce that $J_1(i)=J_2(i)$. 
We again distinguish two cases.\\
{\bf Case 1}: Assume $A_i=\phi_{\ti}$ or $\phid_{\ti}$. Then $\deg(T_{J_1,i})=s_{b_i}-pJ_1(i)+J_1(i+1)$, $\deg(T_{J_2,i})=s_{b_i}-pJ_2(i)+J_2(i+1)$. Since $J_1(i+1)=J_2(i+1)$, we deduce that $J_1(i)=J_2(i)$. \\
{\bf Case 2}: Assume $A_i=\psi_{\ti}$ or $\psid_{\ti}$. Then $\deg(T_{J_1,i})=s_{b_i}-pJ_1(i)+r-J_1(i+1)$, $\deg(T_{J_2,i})=s_{b_i}-pJ_2(i)+r-J_2(i+1)$. Since $J_1(i+1)=J_2(i+1)$, we deduce that $J_1(i)=J_2(i)$.
\end{proof}

\section{Main result}\label{new part mu ordinary}
The goal of this section is to prove Theorem~\ref{abelian cover}, and to highlight some of its applications.
In the following, $\gra$ is an abelian monodromy datum, with $r\leq 5$, and $p$ is a rational prime, $p > |G|(r-2)$.
  
Recall that, given a Frobenius orbit $\OO$ in $\tg$, we denote by 
$f_{\OO,1}<f_{\OO,2}<\dots< f_{\OO,s(\OO)}$  the distinct values in $\FF(\OO)=\{\fts\mid \tau\in \OO\}$. 

\begin{lemma}\label{smallest signature}
Given a Frobenius orbit $\OO$, if $\tau\in \OO$ satisfies $f(\tau^*)=f_{\OO,1}$, then $w_{\tau}$ is generically maximal.
\end{lemma}
\begin{proof}
We apply Theorem \ref{numerical criteria} with $u=1$. It suffices to check that generically $$\dim(\pi_{\tau}\circ L_{\tau, l-1} \circ \dots \circ L_{\tau,0}(M_{\tau}))=f_{\OO,1}.$$ 
Note that $L_{\tau,i}=F_{p^i\tau}$, hence the condition is reduced to showing that $\pi_{\tau} \circ F^{l}: Q_{\tau} \to Q_{\tau}$ is an isomorphism. The main Theorem in \cite[Section 6]{Irene} shows that the determinant of this map is a non-zero polynomial in $\fpb[x_1, \dots, x_r]$. 
\end{proof}

\begin{corollary}\label{Largest signature}
Given a Frobenius orbit $\OO$, if $\tau\in \OO$ satisfies $f(\tau^*)=f_{\OO,s(\OO)}$, then $w_{\tau}$ is generically maximal.
\end{corollary}
\begin{proof}
If $f(\tau^*)=f_{\OO,s(\OO)}$, then $f((\tau^*)^*)=f_{\OO^*,1}$. Therefore, by Lemma \ref{smallest signature}, we know that $w_{\tau^*}$ is maximal. Hence $w_{\tau}$ is also generically maximal.
\end{proof}

\begin{proof}[Proof of Theorem \ref{abelian cover}]

 By combining Lemma \ref{equivalence two} 
and \cite[Theorem 1.3.7]{Moonen EO type formula}, it suffices to show that that $\wt$ is generically maximal, for each $\tau\in \tg$ satisfying $\ker(\tau)=\{1\}$.

Let $\OO$ be a Frobenius orbit.
Recall $s(\OO)$ denotes the number of distinct values in 
 $\FF(\OO)$, and  $g(\OO)=f(\tau)+f(\tau^*)$  for all $\tau\in \OO$. In particular, $s(\OO)\leq g(\OO)+1$.

Assume any/all $\tau \in \OO$ satisfy $\ker(\tau)=\{1\}$.
Then, $g(\OO)=r-2$.
If $r\leq 3$, the statement holds trivially. We distinguish the cases $r=4 $ and $r=5$.

If $r=4$,
then $g(\OO)=2$ and $s(\OO)\leq 3$.
 If $s(\OO)=1$, $\FF(\OO)=\{f_{\OO,1}\}$, then any $\tau\in\OO$ satisfies $f(\tau^*)=f_{\OO,1}$ and $w_\tau$ is generically maximal by Lemma \ref{smallest signature}. 
 If $s(\OO)=2$, $\FF(\OO)=\{f_{\OO,1},f_{\OO,2}\}$, $w_\tau$ is generically maximal by Lemma \ref{smallest signature} if $f(\tau^*)=f_{\OO,1}$ and by Corollary \ref{Largest signature} if $f(\tau^*)=f_{\OO,2}$. 
 If  $s(\OO)=3$, then $\FF(\OO)=\{0,1,2\}$ and $w_\tau$ is generically maximal by Theorem \ref{signature 1} if 
    $f(\tau^*)\in\{0,1\}$, and  
     by Corollary \ref{Largest signature} if  $f(\tau^*)=2$, .


If $r=5$, then $g(\OO)=3$ and $s(\OO)\leq 4$.
If $s(\OO)=1$, the statement follows from Lemma \ref{smallest signature}; 
  if $s(\OO)=2$, it follows from Lemma \ref{smallest signature} and Corollary \ref{Largest signature}.
  If $s(\OO)=3$,  then there are four possibility for $\FF(\OO)$. If $\FF(\OO)=\{0,1,2\}$ or $\FF(\OO)=\{0,1,3\}$, the statement follows direcltly from Theorem \ref{signature 1} and Corollary \ref{Largest signature}.
  If $\FF(\OO)=\{0,2,3\}$ or $\FF(\OO)=\{1,2,3\}$, then $\FF(\OO^*)=\{0,1,3\}$ or $\FF(\OO^*)=\{0,1,2\}$ respectively,
    and the statement follows from the previous instances, since $w_{\tau}$ is maximal if and only if $w_{\tau^*}$ is maximal (see Remark \ref{rmk}).
    If $s(\OO)=4$, then $\FF(\OO)=\{0,1,2,3\}$, and by Theorem \ref{signature 1}, either $w_{\tau}$ or  
     $w_{\tau^*}$ is generically maximal, hence they both are by Remark \ref{rmk}. \qedhere


\end{proof}

\begin{remark}  
The restriction $r\leq 5$ does not imply a bound on the genus of the covers. That is, our result applies to families of curves with arbitrarily large genus. For example, given a cyclic monodromy datum $\lra$, if all $a(j)$ are co-prime to $m$, then 
the genus of the curves parameterized by $\mlra$ is:
$$g=g(m,r,\au)=1 +
\frac{(r-2)m-\sum_{j=1}^r \gcd(a(j),m)}{2}= 1+\frac{(r-2)m-r}{2}.$$
In particular,  $g\lra$ can grow linearly with $m$. 

Also, while the restriction $r\leq 5$ implies $\dim(\mlra)=r-3\leq 2$, it does not imply a bound on the dimension of the Deligne--Mostow Shimura variety $\shgf$. That is, our result applies to families of curves in ambient Shimura varieties of arbitrarily large dimension.   For example,
    consider the cylic monodromy datum $(m,4,(1,1,1,m-3))$, where $3 \nmid m$. Then $\dim(\mlra)=1$, while  
    $\dim(\shmf)\geq \floor{\frac{m+1}{3}}$, which grows linearly with $m$. 

\end{remark}

\begin{remark}\label{newNP}
Theorem \ref{abelian cover} provides new examples of Newton polygons occurring for smooth curves (e.g., polygons with slopes with large denominators).
For example, consider the monodromy datum 
$(23,5,(1,1,1,2,18))$ and a prime $p$ that is inert in $\Q(\zeta_{23})/\Q$. Then, the corresponding cyclic covers of $\Po$ have genus $g=33$, and the associated $\mu$-ordinary Newton polygon at $p$ is $u=(\frac{2}{11},\frac{9}{11})^{2} \bigoplus (\frac{1}{2}, \frac{1}{2})^{11}$.
By Theorem \ref{abelian cover}, for $p> 69$, there exists a smooth curve over $\fpb$ with Newton polygon $u$.
\end{remark}

\begin{remark}\label{anyrl} 
For any $r>5$ and any $m>1$, Theorem \ref{abelian cover}, together with \cite[Proposition 4]{clutching argument}, yields examples of cyclic monodromy data, of degree $m$ and with $r$ branched points, that satisfy the statement  of Theorem \ref{abelian cover}. 
More precisely, in \cite[Proposition 4]{clutching argument}, the authors show that if $\gamma=\lra$ is a cyclic monodromy datum satisfying the statement of Theorem \ref{abelian cover} then, for any $1\leq c\leq m-1$, the monodromy datum $\gamma_c=(m, r+2, {\au}_c)$, where ${\au}_c=(c,m-c,\au_1, \dots, \au_r)$,  also satisfy the statement. Hence, by Theorem \ref{abelian cover}, starting with a cyclic monodromy datum $\lra$ with $4\leq r\leq 5$, one can construct an infinite inductive system of cyclic monodromy data, of degree $m$ and with $r+2t$ branched points, for $t\geq 1$, that satisfy the statement of Theorem \ref{abelian cover}. 
Similarly, one can combine Theorem \ref{abelian cover} with other results in \cite[Section 4]{clutching argument}
to obtain different infinite systems of monodromy data, with an arbitrarily large number of branched points, that satisfy the statement of Theorem \ref{abelian cover} (see Remark \ref{unlikely} for another example of such a construction).
\end{remark}

\begin{remark}\label{unlikely}
Theorem \ref{abelian cover} combined with results in \cite{clutching argument} 
yields new instances of unlikely intersections of the Torelli locus with Newton strata of the Siegel variety.  Following \cite[Definition 8.2]{clutching argument}, we say that the Torelli locus has an {\em  unlikely intersection at $p$} with the Newton polygon stratum $\ag[\nu]$ in $\ag$  if there exists a smooth curve over $\fpb$ of genus $g$ with Newton polygon $\nu$ and $\dim \mg < \text{codim}(\ag[\nu],\ag)$. 
 For example, given a cyclic monodromy datum $\gamma=(m, r, \au)$ with $4\leq r\leq 5$, by Theorem \ref{abelian cover}, \cite[Corollary 4.14]{clutching argument} produces an infinite inductive system of cyclic monodromy data $\{ \gamma_n=(m, r_{n}, \au_{n})\mid n\geq 1\}$,  with $r_{n}= n(r+2)$,  that satisfy the statement of Theorem \ref{abelian cover}. Furthermore, for $p>m(r-2)$, if the $\mu$-ordinary Newton polygon at $p$ associated with $\gamma$ is not ordinary, then by \cite[Proposition 8.5]{clutching argument}, the inductive system yields (infinitely many)  unlikely intersections at $p$ of the Torelli locus with  Newton strata in Siegel varieties. More precisely, if $u$ denotes the $\mu$-ordinary Newton polygon at $p$ associated with $\gamma$, for $p>m(r-2)$, and $u\neq \ord^g$, then  once $n$ is sufficiently large, the Torelli locus has an unlikely intersection at $p$ with the Newton stratum $\ag[u^n+\ord^{n(m-1)}]$.
Continuing the example in Remark \ref{newNP}, for $\gamma=(23, 5, (1,1,1,2,18))$, $p>69$ a prime inert in $\Q(\zeta_{23})/\Q$, $g=33$ and $u=(\frac{2}{11},\frac{9}{11})^{2} \bigoplus (\frac{1}{2}, \frac{1}{2})^{11}$, by \cite[Formula 8.1]{clutching argument},  $\dim(\mg)= 96<\text{codim}(\ag[u],\ag)=136$ and the Torelli locus has an unlikely intersection at $p$ with the Newton stratum $\ag[u]$. Furthermore, by \cite[Proposition 8.4]{clutching argument}, for any $n\geq 1$, the Torelli locus has an unlikely intersection at $p$ with the Newton stratum $\mathcal{A}_{g_n}[u^n+\ord^{n(m-1)}]$, where $g_n=n(m+32)$. 
\end{remark}

\subsection*{Acknowledgements}
\thanks{

We would like to thank Rachel Pries for many helpful discussions that have motivated the pursuit of this problem. Mantovan was partially supported by NSF grant DMS-22-00694.
}

\end{document}